\documentclass[11pt, twoside]{article}
\usepackage{amsfonts,amssymb,amsmath,amsthm}
\usepackage{verbatim}
\usepackage{graphicx}
\usepackage[all]{xy}

\usepackage{psfrag,xmpmulti,amscd,color,pstricks, import}

 \divide\oddsidemargin by 2
\newtheorem{thm}{Theorem}[section]
\newtheorem{cor}[thm]{Corollary}
\newtheorem{lem}[thm]{Lemma}
\newtheorem{prop}[thm]{Proposition}

\theoremstyle{definition}
\newtheorem{defn}[thm]{Definition}
\newtheorem{rem}[thm]{Remark}
\newtheorem*{rem*}{Remark}
\newtheorem*{mainthm}{Main Theorem}

\numberwithin{equation}{section}

\newtheorem{subclaim}{Claim}

\definecolor{OrangeRed}{cmyk}{0,0.6,1,0}            % half magenta only, full yellow
\definecolor{DarkBlue}{cmyk}{1,1,0,0.20}
\definecolor{DarkGreen}{cmyk}{1,0,0.6,0.2}
\definecolor{myblue}{rgb}{0.66,0.78,1.00}
\definecolor{Violet}{cmyk}{0.79,0.88,0,0}
\definecolor{Lavender}{cmyk}{0,0.48,0,0}

\newcommand{\diam}{\operatorname{diam}}
\newcommand{\eucl}{\operatorname{eucl}}

\newcommand{\dist}{\operatorname{dist}}

\renewcommand{\AA}{{\cal A}}

\newcommand{\BB}{{\cal B}}

\newcommand{\FF}{{\cal F}}

\newcommand{\LL}{{\cal L}}

\newcommand{\PP}{{\cal P}}

\renewcommand{\SS}{{\cal S}}
\newcommand{\TT}{{\cal T}}

\newcommand{\VV}{{\cal V}}
\newcommand{\XX}{{\cal X}}

\newcommand{\C}{{\mathbb C}}

\newcommand{\D}{{\mathbb D}}

\newcommand{\N}{{\mathbb N}}

\newcommand{\R}{{\mathbb R}}

\newcommand{\ra}{\rightarrow}
\newcommand{\ov}{\overline}

\renewcommand{\epsilon}{\varepsilon}
\renewcommand{\phi}{\varphi}

\newcommand{\s}{{\underline{s}}}

\newcommand{\Brays}{\BB_{\operatorname{rays}}}

\usepackage{accents}

\DeclareRobustCommand{\Bacc}[1]{\accentset{B,r}{#1}}

\newcommand{\compl}{\C\setminus (\ov{\D}\cup\delta)}

 \title{A bound on the number of rationally invisible repelling orbits}
 
\vspace{5cm}
\author{Anna Miriam Benini \thanks{This project has received funding from the European Union's Horizon 2020 research and innovation programme under the Marie Sk\l odowska-Curie Grant Agreement No. 703269   COTRADY} {\small\ and}  N\'uria Fagella \thanks{Partially supported by the Spanish 
grant MTM2017-86795-C3-3-P, the Maria de Maeztu Excellence Grant MDM-2014-0445 of the BGSMath and the catalan grant 2017 SGR 1374.} \\
\small Dept.~de Matem\`atica y Inform\`atica, Universitat de Barcelona\\ \small Barcelona Graduate School of Mathematics (BGSMath) \\ %\addressbreak
\small Gran Via 585 08007, Barcelona, Catalonia}

\begin{document}

\maketitle  
\begin{abstract}
  We consider entire transcendental maps with bounded set of singular values such that  periodic rays exist and land. For such maps, we prove a refined version of the Fatou-Shishikura inequality which takes into account rationally invisible periodic orbits, that is, repelling cycles which are not landing points of any periodic ray. More precisely, if there are $q<\infty$ singular orbits, then the sum of the  number of attracting, parabolic, Siegel, Cremer or rationally invisible orbits is bounded above by $q$. In particular, there are at most $q$ rationally invisible repelling periodic orbits. The  techniques presented here  also apply to the   more general setting  in which the function is allowed to have infinitely many singular values.  
\end{abstract}

\section{Introduction}

Consider an entire transcendental function $f$ and let $S(f)$ be its  set of \emph{singular values} 

$$S(f):=\ov{\{\text{asymptotic values, critical values}\}}.$$

Then $f:\C\setminus f^{-1}(S(f))\ra\C\setminus S(f) $ is an unbranched covering of infinite degree.
The closure of the orbits of all singular values is called the \emph{postsingular set} and is denoted by $$\PP(f):=\ov{\bigcup_{s\in S(f), n>0}f^n(s)}.$$ A singular value $v$ is \emph{non-recurrent} if it does not belong to its $\omega$-limit set  $\omega(v)$, defined  as the set of accumulation points for the orbit $\{f^n(v)\}_{n\in\N}$.  
 
Many of the intricate patterns that arise in the dynamics of holomorphic maps are due to the presence of singular values and to the way in which their orbits interact with each other. For example, it is not difficult  to show that if the unique  singular value of a quadratic polynomial is  non-recurrent then the Julia set is locally connected \cite[Expos\'e X]{DH84}. Similarly, the presence of non-repelling periodic orbits is entangled with the behavior of singular orbits. For example, every immediate attracting or  parabolic basin needs to contain  a singular orbit, and each Cremer point or  point in the boundary of a Siegel disk needs to be accumulated by points in the postsingular set \cite{Fa20,Mi}.   

As a consequence of this deep relationship, it is possible to give an upper bound for the number of non-repelling cycles  in terms of the number of singular values. This is known as the {\em Fatou-Shishikura inequality} \cite{Sh87, EL92}, and it states that if an entire map (polynomial or transcendental) has $q$ singular values, then 
\[
 N_{\text{non-repelling}}  \leq q,
 \]
where  $N_{\text{non-repelling}}$ stands for the number of attracting, parabolic, Cremer and Siegel cycles. The proof of this celebrated result relies on perturbations in parameter space. However, with   additional dynamical assumptions on the map (for example, bounded postsingular set), a more combinatorial approach in the dynamical plane also associates each non-repelling cycle to a singular orbit in a precise mathematical way, and in such a way that the latter cannot be associated to any other non-repelling cycle \cite{Ki00,BF17}. 

To be somewhat more precise on this extra assumption, we must talk about {\em rays}  \cite{DH84,Mi,RRRS}. For polynomials, and for many transcendental maps the \emph{escaping set}, defined as 
$$
I(f):=\{z\in\C;\  f^n(z)\ra\infty\},
$$ 
consists of injective, mutually disjoint curves $G:(0,\infty)\ra I(f)$ tending to infinity as $t\to \infty$. These are called \emph{external rays} for polynomials and  \emph{dynamic rays} (or \emph{hairs}) for transcendental maps (see Section~\ref{sect:Background}  for a precise definition), although in this paper we will often call them just \emph{rays}. A ray $G$   is \emph{periodic} if $f^n(G)\subset G$ for some $n\in\N$, and we say that it  \emph{lands} at a point $z_0\in\C$ if $G(t)\ra z_0$ as $t\ra 0$.   Periodic rays can only land at parabolic or repelling periodic points by the Snail Lemma \cite{Fa20,Mi}.  

Polynomial rays foliate the attracting basin of infinity and hence lie in the Fatou set. Their landing is tightly related to the topology of the Julia set. Indeed, the Julia set is locally connected if and only if  all rays land, in which case the Julia set can be parametrized by the unit circle.  For transcendental maps, the situation is more complex. To start with, it is not always true that the escaping set is formed by rays, although this is the case for a wide class of entire transcendental functions \cite{RRRS,BRG17}. This class includes  the class $\Brays$ of functions which are finite compositions of functions of finite order with bounded set of singular values, and for such functions the escaping set lies entirely in the Julia set.

In the case of polynomials or maps in $\Brays$, if the postsingular set is bounded all periodic rays land (at repelling or parabolic) periodic points \cite{DH84,Hu,Mi,Re08,De}.  Conversely, one may ask whether every repelling or parabolic point is the landing point of a ray or, in other words, whether repelling and parabolic points are always accessible from the escaping set.  The answer to this question is not always positive and motivates the following definition.
\begin{defn}[Rationally invisible periodic orbit]
A repelling periodic orbit of an entire map (polynomial or transcendental) is called \emph{rationally invisible} if one of the points in the orbit (and hence all of them) is not the landing point of any periodic ray. 
\end{defn}

The non-existence of rationally invisible periodic orbits, whenever it can be proven, has consequences for the study of parameter spaces.  In polynomial dynamics, for example, it represents the starting point for Yoccoz puzzle and for much of the machinery which lead to most of the actual rigidity results. In transcendental dynamics, it is related to the non-existence of ghost limbs attached to hyperbolic components. It is therefore of interest to understand the situations under which these special orbits may exist. 

As it turns out, rationally invisible orbits, despite being repelling, are also tightly related to the orbits of the singular values and more precisely, to unbounded singular orbits. Indeed, if an entire map (polynomial or transcendental in $\Brays$) has a bounded postsingular set, then every repelling or parabolic periodic point is the landing point of at least one and at most finitely many periodic rays, and hence there are no rationally invisible orbits (see \cite{DH84,Hu,Mi} for polynomials, \cite{BF15,BL14, BRG17} for transcendental). 

In the absence of this boundness restriction, one would like to give an upper bound for the number of rationally invisible periodic orbits in terms of the number of singular values, so as to produce a refinment of the Fatou-Shishikura inequality. This is indeed the case for polynomials \cite[Corollary 1]{LP}, \cite{BCLOS16}, and also for transcendental maps, as we show in the main result of this paper (see also Theorem~\ref{thm:main theorem for fp} for a stronger statement).

As usual we say that a singular value escapes along periodic rays if its orbit converges to infinity and eventually belongs to a cycle of periodic rays.
  
\begin{mainthm}\label{thm:main thm intro} Let $f\in\Brays$ such that periodic rays land and assume that there are no singular values escaping along  periodic rays. Suppose that $f$ has at most $q<\infty$ singular orbits which do not   belong to attracting or parabolic basins. 

Let   $N_{\text{indifferent}}$ denote  the number of Cremer cycles and cycles of Siegel disks,   and $N_{\text{invisible}}$ denote the number of  rationally invisible orbits. Then we have  
$$ 
N_{\text{indifferent}} +N_{\text{invisible}} \leq q .
$$
In particular,  there are at most $q$ rationally invisible repelling periodic orbits.
\end{mainthm}

One may wonder about how strong is the assumption that periodic rays land in the transcendental setting. For polynomials, this assumption is equivalent to the assumptions that no critical points escape along  periodic external rays and is implied by the standard assumption of the Julia set being connected.  
It is expected that also in the transcendental case a periodic ray lands unless its forward orbit contains a singular value.  This has been proven for  functions in the exponential family using   parameter space based arguments \cite{Re06}, which seem to be  out of reach even for functions with finitely many singular values.  
The assumption that there are no singular values escaping along periodic rays  is evidently weaker that the 
hypothesis that periodic rays  do not intersect  the postsingular set. In fact, the latter hypothesis implies landing  of periodic rays \cite{Re08}. 

The Main Theorem has the following immediate corollary.
\begin{cor}Let $f\in\Brays$ such that periodic rays land and assume that there are no singular values escaping along  periodic rays.  Suppose that $f$ has at most $q$ singular orbits which do not  belong to attracting or parabolic cycles. Then there are at most $q$ repelling periodic orbits which are rationally invisible. 
\end{cor}

The only previous known result in the direction of putting a bound on the number of rationally invisible periodic orbits of transcendental maps is due to Rempe-Gillen \cite{Re06}, and states that for  any $f_c(z)=e^z+c$ there is at most one rationally invisible periodic orbit. The proof uses arguments in the parameter space of the exponential family and relies  crucially on the existence and structure of  wakes  in the parameter plane. 

Instead, the proof that we present in this paper uses the  structure of the dynamical plane carved by periodic rays \cite{BF15}, and is of a local nature. As a bonus, it also gives more information about the accumulation behavior of the singular orbits (see Theorem~\ref{thm:main theorem for fp}).

As a concluding remark, let us note that in exponential dynamics, for parameters for which the postsingular set is bounded  there are no rationally invisible repelling periodic orbits, and this fact  implies that such parameters cannot belong to ghost limbs attached to hyperbolic components (See Theorem 4, the final conjecture in \cite{Re06}, and the last section in \cite{BL14}). This  has also  been used for some of the rigidity results in \cite{Be15}. This type of results increase our current knowledge of parameter spaces. For families of transcendental functions with more than one singular value this knowledge  is currently very limited, but there is no doubt that it will undergo an important development in the next decades. We hope that  the  results and the techniques developed in this paper will be a little brick in the implementation  of this large project.

The paper is structured as follows. Section 2 contains the background about functions in class $\Brays$ and their combinatorics and presents the Separation Theorem \cite{BF15}, a key tool for the proof of the main result. It describes also fundamental tails, objects introduced in \cite{BRG17} which can be seen as intermediate steps in the construction of rays, and which despite their intricate combinatorics, have proven to be useful in the proof of several recent results. In Section 3 we give a characterization of landing of periodic rays in terms of some combinatorics of tails. 

Meanwhile, Section~\ref{sect:main theorem and proof} contains the statement and the proof Theorem~\ref{thm:main theorem for fp}, from which the Main Theorem follows, a relation which is made explicit in Section 5. 
  Section~\ref{sect:main theorem and proof} contains also a corollary (see Corollary~\ref{cor:fibers}) stating that, under our assumptions,  the union of the  dynamical   fibers (as in the definition of \cite{RSbif}) of  a rationally invisible repelling periodic orbit contains  either a singular orbit, or infinitely many singular values whose orbits belong to the fiber for more and more iterations.

\subsection*{Notation}
Let $\C$ denote the complex plane, $\D$ the unit disk. The Euclidean disk of center $z$ and radius $r$ is denoted by $\D_r(z)$.   
By a \emph{(univalent) preimage under $f^n$} of an open connected set $V$  we mean a connected component  $U$ of the set $f^{-n}(V)$ (such that $f^n:U\ra V$ is univalent). Given a set $A$ and $k\in\N$ we denote by $\{A\}^k$ the set $A\times\ldots\times A$ where the product is taken $k$ times. 

\subsection*{Acknowledgments} We are thankful to Misha Lyubich for enlightening discussions and to the Centro di Ricerca Matematica Ennio de Giorgi in Pisa for their hospitality, under the Research in Pairs program.
 
\section{Background}\label{sect:Background}
\subsection*{Tracts, fundamental domains, and dynamic rays}
  
Let $f$ be an entire transcendental function with bounded set of singular values and let $D$ be a Euclidean disk  containing $S(f)$ and $f(0)$.  The connected components of $f^{-1}(\C\setminus \ov{D})$ are called   \emph{tracts} \cite{EL92} and  are unbounded and simply connected. By definition for any tract $T$ we have that $f:T\ra\C\setminus \ov{D}$ is an unbranched covering of infinite degree. Let $\TT$ be the union of all tracts.
 It is not difficult to find  an   analytic curve $\delta\subset \C\setminus(\ov{D}\cup \TT)$ connecting $\partial D$ to $\infty$ (\cite{Rottenfusserthesis}; see also \cite[ Lemma 2.1]{BF15}). Let $\Omega:=\C\setminus (\ov{D}\cup\delta)$. The connected components of $f^{-1}(\Omega)$ are called \emph{fundamental domains}.   It is easy to see that only finitely many fundamental domains intersect $D$ and that for any fundamental domain $F$ we have that $f:F\ra\Omega$ is a biholomorphism. We denote by $\FF$ the collection of all fundamental domains, as well as their union. 
 
The structure of the dynamical plane given by tracts and fundamental domains has been useful to construct dynamic rays.  
The initial idea of finding curves in the escaping set of transcendental entire functions goes back to \cite{Fa26}, was later developed  in \cite{DT86}, \cite{DGH}, \cite{DK}, \cite{BK07}, \cite{Ba07},   \cite{RRRS} among others.

\begin{defn}[\bf Dynamic ray]
A \emph{(dynamic) ray} for $f$ is an injective curve  $G:(0,\infty)\ra I(f)$ such that:
\begin{itemize}
\item[(a)] $\underset{t\ra\infty}\lim |f^n(G(t))|=\infty\,\  \forall n\geq 0$;
\item[(b)] $\underset{n\ra\infty}\lim |f^n(G(t))|=\infty$  uniformly in $[t_0,\infty)$ for all $t_0>0$;
\item[(c)] $f^n(G(t))$ is not a critical point for any $t>0$ and $n\geq0$;
\end{itemize} 
and such that $G(0,\infty)$ is maximal with respect to these properties. If  $G(0,\infty)$ is  maximal with respect to (a) and (b) but not with respect to (c),  then we call the ray {\em broken}.
\end{defn}

Broken rays could therefore be continued if we allowed critical points and their iterated preimages to be part of the ray, as it is the case in the definition in \cite{RRRS}, where branching might occur and several rays might share one same arc. This  situation cannot happen in our setting, i.e.  rays are pairwise disjoint. 
 
A  dynamic ray $G$ is  \emph{periodic} if $f^p(G)= G$ for some $p\geq 1$, and \emph{fixed} if $p=1$. We say that a dynamic ray \emph{lands} at a point $z_0\in\C$ if  it is not broken and  $\lim G(t)= z_0$ as $t\ra 0 $. Observe that dynamic rays are allowed to land at singular values, but that broken rays are  not considered to  land. 

Recall that  $\Brays$ denotes the class of transcendental entire functions which are finite compositions of functions of finite order with bounded set of singular values. 
In \cite[Theorem 1.2]{RRRS} it is shown that for any $f\in\Brays$ and for any escaping point $z$ then $f^n(z)$ belongs to a dynamic ray for any $n$ large enough. 
 For this paper we need to take into account a combinatorial description of dynamic rays, which is implicitly contained in \cite{RRRS} and in several of the aforementioned papers but for which we use the explicit setup that has been presented in \cite{BF15}.

We say that a dynamic  ray $G$ is \emph{asymptotically contained} in a fundamental domain $F$ if $G(t)\in F$ for all $t$ sufficiently large. It is easy to see that this is always the case, as stated in the following lemma.

\begin{lem}[See e.g. Lemma 2.3 in \cite{BF15}]\label{lem:asymptotically contained} Let $f\in \Brays$. Then every dynamic ray is asymptotically contained in a fundamental domain.
\end{lem} 

Let us consider the symbolic space formed by all infinite sequences of fundamental domains
\[\FF^\N=\{\s=F_0 F_1 F_2\ldots\}\]
endowed with the dynamics of the shift map $\sigma:\FF^\N\ra\FF^\N$, $\sigma F_0 F_1 F_2\ldots= F_1 F_2 F_3\ldots$. For $\s=F_0F_1\ldots\in\FF^\N$, the set  $\sigma^{-1}\s$  of its  preimages is given by all sequences of the form  $F\s:=F F_0 F_1\ldots$ where $F\in\FF$.  

\begin{defn}We say that a dynamic ray $G$ has \emph{address} $\s=F_0F_1\ldots\in\FF^\N$ and we denote it by $G_{\s}$ if and only if $f^j(G_\s)$ is asymptotically contained in $F_j$ for all $j$.
\end{defn}
It follows directly from the construction in \cite{RRRS} that given an address $\s$ the ray $G_\s$, if it exists, is unique, and that for  rays which are not broken we have that 
\[f(G_\s)=G_{\sigma\s}\] and that 
\[\{f^{-1}G_\s\}=\{G_{F\s}:F\in\FF \}.\]
This implies  that a  dynamic ray $G_\s$ is periodic if and only if  $\s$ is periodic. We say that $G_\s$ has \emph{ bounded address} if $\s$ is bounded, i.e. its entries  take values over finitely many fundamental domains. 
 
 The next proposition is  \cite[Proposition 2.11]{BF15}, where it is proven using results and ideas  from \cite{DT86} and \cite{RRRS}. It previously appeared in different formulations in \cite{Re08}, \cite{BK07}.
  
\begin{prop}\label{Existence for finitely many symbols}
If $f\in\Brays$ and $\s\in\FF^\N$ is bounded  then there exists  a  unique dynamic ray $G_\s$ with address $\s$ for $f$.
\end{prop}
\begin{rem} A  generalization of rays for functions with not as beautiful a  geometry as functions in class $\Brays$ can be found in \cite{BRG17}. The unbounded, connected sets which take the place of rays   are called \emph{dreadlocks}. Despite not being curves, dreadlocks have the same combinatorial structure as rays. The results that are presented for rays in this section also hold for dreadlocks. 
\end{rem}
 \subsection*{The Separation Theorem}
 
Goldberg and Milnor  \cite{GM} proved that for polynomials with connected Julia set, the set of fixed rays together with their landing point separate the set of   fixed points which are not landing points of fixed rays; such points include all attracting, Siegel and Cremer parameters. 

 Goldberg-Milnor's  theorem has been generalized  to entire transcendental maps in class $\Brays$, under the assumption that periodic rays land \cite{BF15}. In order to state the theorem we need to  introduce the notion of \emph{basic regions} and \emph{interior fixed points}, following \cite{GM} and \cite{BF15}.  
Fix $p$ and assume that periodic dynamic rays land. Let $\Gamma$ denote the closed  graph formed by the rays  fixed by $f^p$ together with their landing points.  The connected components of $\C\setminus \Gamma$ are called the \emph{basic regions for $f^p$.} An \emph{interior fixed point} for $f^p$ is a periodic point for $f$ which is fixed by $f^p$ and which is not the landing point of any periodic ray which is fixed by $f^p$. Note that attracting, Siegel and Cremer points as well as rationally invisible repelling periodic points are interior periodic points for $f^p$ for all $p$, while parabolic and repelling periodic points may be interior or not depending on $p$. For example a   fixed point which is the landing point of a cycle of periodic rays of period $3$ is interior for $f$ but not for $f^3$.

\begin{thm}[Separation Theorem Entire \cite{BF15}]\label{thm:Separation Theorem Entire} Let  $f\in\Brays$, $p\in\N$ and assume  that all periodic  rays for $f$ which are fixed by  $f^p$ land. Then there are finitely many basic regions for $f^p$, and each basic region  contains exactly one interior fixed point for $f^p$, or exactly one attracting parabolic basin which is invariant under $f^p$.
\end{thm}
 
Theorem~\ref{thm:Separation Theorem Entire} has many corollaries, including that parabolic points are always landing points of periodic dynamic rays (whose period equals the period of the attracting basins), and that  hidden components of a Siegel disk are preperiodic to the Siegel disk itself (see \cite{CR}, \cite{BF2} for an application of this fact to the existence of critical points on the boundary of Siegel Disks).  It has recently been used in \cite{BF17} to associate non-repelling cycles to singular orbits under the hypothesis that periodic rays land.

\subsection*{A couple of useful lemmas}

The following two general lemmas will be used several times  in the sequel. 
The first of them is  Lemma 2.1 in \cite{BRG17}
\begin{lem}\label{lem:finitely many preimages intersect compact sets}Let $f:\C\ra\C$ holomorphic, $U\subset\C$ be  a connected set with locally connected boundary. Then for any compact set $K\in\C$, only finitely many connected components of $f^{-1} (U)$ intersect $K$.
\end{lem}
 
 \begin{lem}[Forward invariant boundary]\label{lem:forward invariant boundary} Let $f$ be holomorphic, $B$ be a region whose boundary is forward invariant, $V$ be an open subset of  $\C$ which does not intersect the boundary of $B$. Then for any connected component $U$ of $f^{-1}(V)$ we have that $U$ is either contained in $B$ or in $\C\setminus \ov{B}$.
 \end{lem} 
 
 \begin{proof}
 Otherwise, $U\cap\partial B\neq \emptyset$. Since $f(\partial B)\subset \partial B$
 it follows that $V\cap\partial B\neq \emptyset$ contradicting the hypothesis. \end{proof}

\section{Fundamental tails for a repelling periodic orbit}\label{sect:Fundamental tails for a repelling periodic orbit}
 
Fundamental tails are relatively new objects introduced in \cite{BRG17} for functions with bounded postsingular set. They already found application  in \cite{FagellaTails}. Fundamental tails are     preimages of fundamental domains under finitely many iterates, and hence are nice open sets. Loosely speaking they  can be thought of as approximation of  rays, which  despite being topologically curves, do not necessarily have nice geometric properties. In what follows we give a precise definition of tails under weaker assumptions than in the original setting.

 Let $f\in\Brays$ whose periodic rays land.  
Let $z_0$ be a  repelling periodic point of minimal period $m$ and  let $\XX=\{z_0,z_1\ldots z_{m-1}\}$ be its orbit labeled so that $f(z_i)=z_{i+1}$ with indices taken modulo $m$. Let $p$ be a multiple of $m$. Suppose that $z_0$ is an interior fixed point for $f^p$,  and consider the basic regions $B_0,\ldots B_{m-1}$ for $f^p$ which contain the elements of $\XX$, namely, $z_i\subset B_i$ for $i=0,\ldots, m-1$. indices of the basic regions $\{B_i\}$ will also be taken modulo $m$.

  Let $B$ denote    the union of the $B_i$.  Since  there are only finitely many  basic regions for $f^p$ (see Theorem~\ref{thm:Separation Theorem Entire}),  the boundary of $B$ contains finitely  many pairs of   rays which are fixed under $f^p$, together with their landing points.  Let  $D,\delta$    as in Section~\ref{sect:Background}. Let $\FF_B$ be the collection of fundamental domains intersecting $B$ for $D,\delta$.

Fix   $r>0$ such  that   $r>|z_i|$ for all $i\in \XX$ and let $D_r\supset D$ be the Euclidean disk of radius $r$ centered at $0$ . Let $\delta_r\subset\delta$ be the unbounded connected component of $\delta\setminus D_r$.  
For any fundamental domain $F\in\FF_B$ let $\Bacc{F}$  be the unique unbounded connected component of $F\cap B\cap f^{-1}(\C\setminus( \ov{D_r}\cup\delta_r))$. This is the same as saying that we are considering   $\Bacc{F}$ to  be the unique unbounded connected component of the fundamental domains obtained by using $D_r, \delta_r$ instead of $D,\delta$, intersected with $B$.

 \begin{figure}[hbt!] 
\begin{center}
\def\svgwidth{0.3\textwidth}
%%%%%
\begingroup%
  \makeatletter% 
    \setlength{\unitlength}{\svgwidth}%
  \makeatother%
  \begin{picture}(1,1.28442043)%
    \put(0,0){\includegraphics[width=\unitlength]{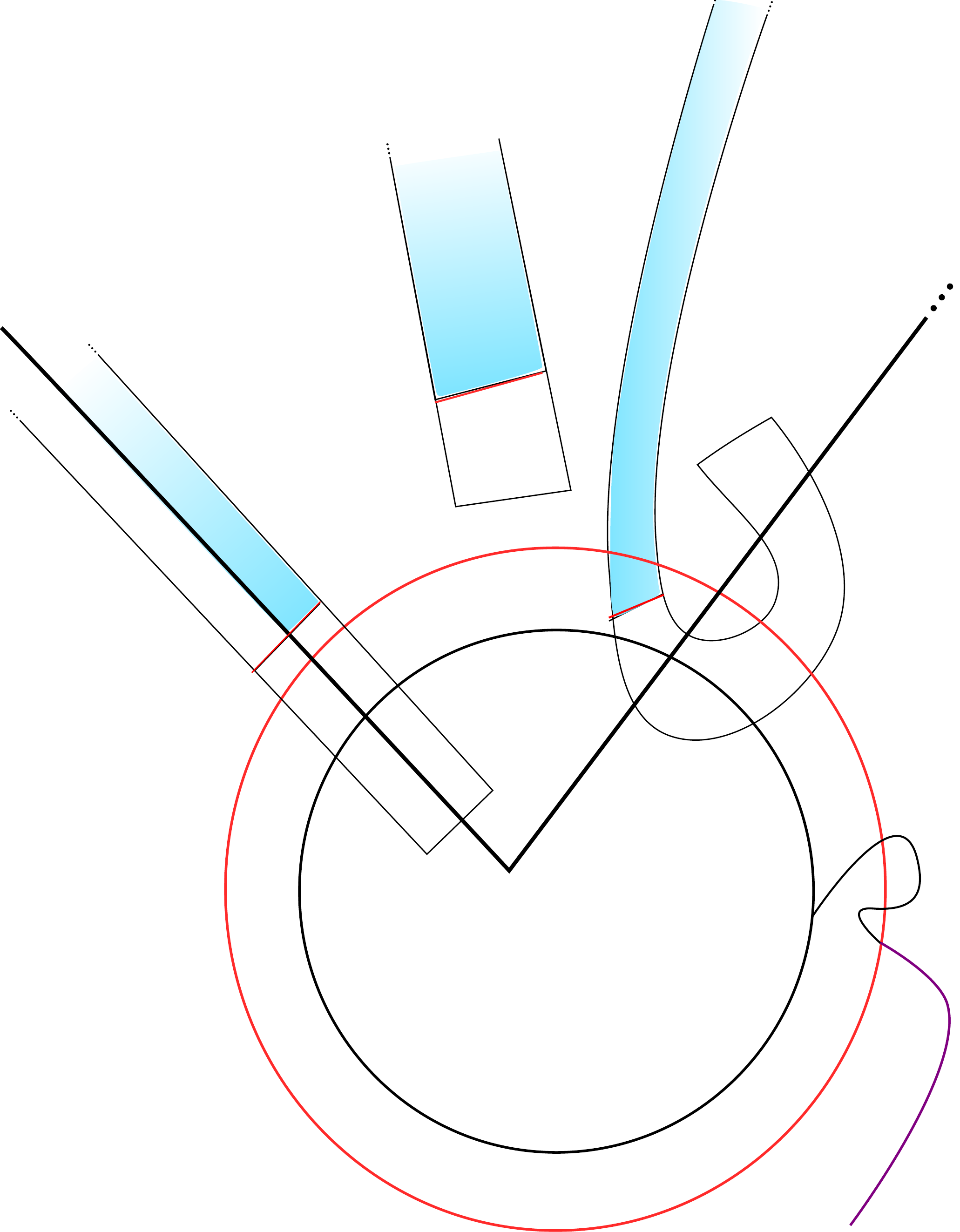}}%
    \put(0.65333534,0.26329139){\color[rgb]{0,0,0}\makebox(0,0)[lb]{\small{$D$}}}%
    \put(0.13,0.25){\color[rgb]{0,0,0}\makebox(0,0)[lb]{\small{$D_r$}}}%
    \put(0.90632464,0.68812244){\color[rgb]{0,0,0}\makebox(0,0)[lb]{\small{$F_1$}}}%
    \put(0.42659965,0.79075016){\color[rgb]{0,0,0}\makebox(0,0)[lb]{\small{$F_2$}}}%
    \put(0.22850426,0.80507033){\color[rgb]{0,0,0}\makebox(0,0)[lb]{\small{$F_3$}}}%
      \put(1,0){\color[rgb]{0,0,0}\makebox(0,0)[lb]{\small{$\delta$}}}%
  \end{picture}%
\endgroup%

 %%%%%%%%%%%%%%%%%%%55
\end{center}
\caption{\small Definition of $\Bacc{F}$ when $B$ is a single basic region. The region $B$ is shown  
  together with the disk $D$ used to define fundamental domains. For simplicity only $3$ fundamental domains $F_1, F_2, F_3$ are shown.  The circle of radius $r$ and its preimages inside  $F_1, F_2, F_3$ are in red. Shaded in light blue are the tails of level 1 for the disk $D_R$ corresponding to $F_1, F_2, F_3$. The curve $\delta_r$ is in purple.}
\label{fig:recedingFD}
\end{figure}
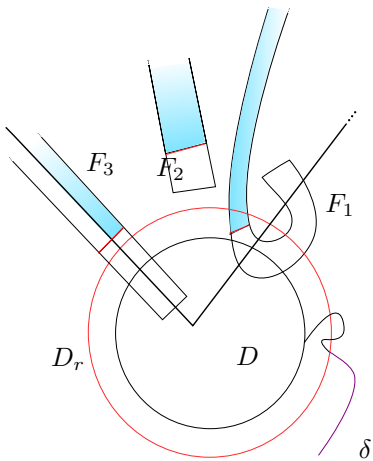

For any $F\in\FF_B$ we have that    $\Bacc{F}$ does not intersect $\partial B$, so by Lemma~\ref{lem:forward invariant boundary} for any $n$ we have that  any  connected component of   $f^{-n}(\Bacc{F})$ is  contained in  either $B$ or $\C\setminus\ov{B}$.

\begin{defn}[Fundamental tails for $z_0$]\label{defn:fundamental tails}
Let $z_0,m, $
The set of tails of level $1$, that we denote by  $\TT_1=\TT_1(r)$, is the set of $\Bacc{F}$ where $F\in\FF_B$ and $\Bacc{F}\subset B_0$.
 Since $D_r\supset D$, we have that $f|_\tau:\tau\ra\C\setminus (\ov{D_r}\cup\delta_r)$ is univalent for any $\tau\in\TT_1$. 
 We define tails of level $n$ by induction. Suppose that we have defined the set  $\TT_{n}$ of  tails of level $n$ and let us define the set  $\TT_{n+1}$ of  tails of level $n+1$.
  
 We say that $\eta$ is a tail of level $n+1$ (for $z_0$, $r$) if it satisfies the following.
 \begin{itemize}
\item $\eta$ is a connected component of  $f^{-m}(\tau)$ for some $\tau\in \TT_n$;
 \item $\eta\subset B_0$ and $f^m:\eta\ra\tau$ is univalent; 
 \item $f^i(\eta)\subset B_i$ for $i=0,\ldots, m-1$.
 \end{itemize}
\end{defn}

It follows that if $\eta\in\TT_n$, $f^{m(n-1)}:\eta\ra\tau$ is univalent, where $\tau$ is some element of $ \TT_1$.

The definition above depends on the choice of $z_0, p $ and  $r$. The point $z_0$, its period $m$, and the period $p$ of the basic regions are fixed throughout the section, while $r$ may vary.   
 With this definition all tails of all levels are contained in the basic region $B_0$ which contains $z_0$, and have the following properties.
\begin{lem}\label{lem:tails for z0 asymptotically contained} Let $\tau$ be a fundamental tail of level $n$ for $z_0,r$. Then:
\begin{itemize}
\item $\tau$ is asymptotically contained in a unique fundamental domain  $F_0\in\FF_B$ which intersects $B_0$, that is, there is a unique fundamental domain $F_0\in\FF_B$ which intersect $B_0$ and such that $\tau\cap F_0$ is unbounded.
\item For $j=1,\ldots, m(n-1)$, $f^j(\tau)$ is asymptotically contained in a fundamental domain $F_j$  which intersects $B_{j}$, that is, there is a unique fundamental domain $F_j$ which intersects $B_j$ and such that $\tau\cap F_j$ is unbounded.
\end{itemize}
\end{lem}
\begin{proof}
The proof follows from the definition of fundamental tails.
\end{proof}
Lemma~\ref{lem:tails for z0 asymptotically contained} gives a way  to dynamically associate a finite sequence of fundamental domains (called an \emph{address}) to each  tail $\tau \in \TT_n$,  similarly to the way in which we associate addresses to dynamic rays. Compare with Definition 3.7 and 3.8 in \cite{BRG17}. 
 
\begin{defn}[Addresses of fundamental tails]\label{defn:addresses of tails} Let $\tau$ be a fundamental tail of level $n$  and let $\s=F_0F_1\ldots F_{m(n-1)}$ be the sequence of fundamental domains given by Lemma~\ref{lem:tails for z0 asymptotically contained}. We say that $\s$ is the (finite) \emph{address} of $\tau$. Observe that $\s$ has length  $\ell_n=m(n-1)+1$.
 When it exists,  we define $\tau_n(\s)$   to be the unique tail of level $n$ and address $\s$. Uniqueness is given by the fact that  for each fundamental domain $F$ we have that  $f: F\ra \compl$ is a homeomorphism.
\end{defn}

At first glance one may expect that all sequences whose elements are fundamental domains intersecting $B$ should be realized. However, some of these fundamental domains are only partially contained in $B$, and this prevents the existence of some  tails. One can characterize precisely the set of addresses which are realized but this is not needed for our purposes.

The set of addresses of tails of level $n$ is contained in  $\{\FF_B\}^{\ell_n}=\FF_B\times \FF_B^{m(n-1)}$.  Consider infinite  sequences in $\{\FF_B\}^\N$. There is a natural projection 
\[\pi_n:\{\FF_B\}^\N\ra \{\FF_B\}^{\ell_n}\]
 which maps an infinite sequence $\s$ to the finite address consisting of its  first $\ell_n$ entries. In this sense, whenever it exists,  we can define the tail of level $n$ and address $\s\in\{\FF_B\}^\N$ as the tail of level $n$ and address $\pi_n(\s)$. We refer to elements in $\{\FF_B\}^\N$ as (infinite) addresses, despite the fact that not all of them are realized as fundamental tails of arbitrarily high levels.

The set of admissible addresses is denoted by 
\begin{equation}\label{eqtn:AB}
\AA_B=\AA_B(z_0, p, r):=\{ \s\in\{\FF_B\}^\N: \text{ the tail $\tau_n(\s)$ is well defined for all $n$}\}.
\end{equation} 

\begin{figure}[hbt!] 
\begin{center}
\def\svgwidth{0.6\textwidth}
%%%%%
\begingroup%
  \makeatletter% 
    \setlength{\unitlength}{\svgwidth}%
  \makeatother%
  \begin{picture}(1,0.60525985)%
    \put(0,0){\includegraphics[width=\unitlength]{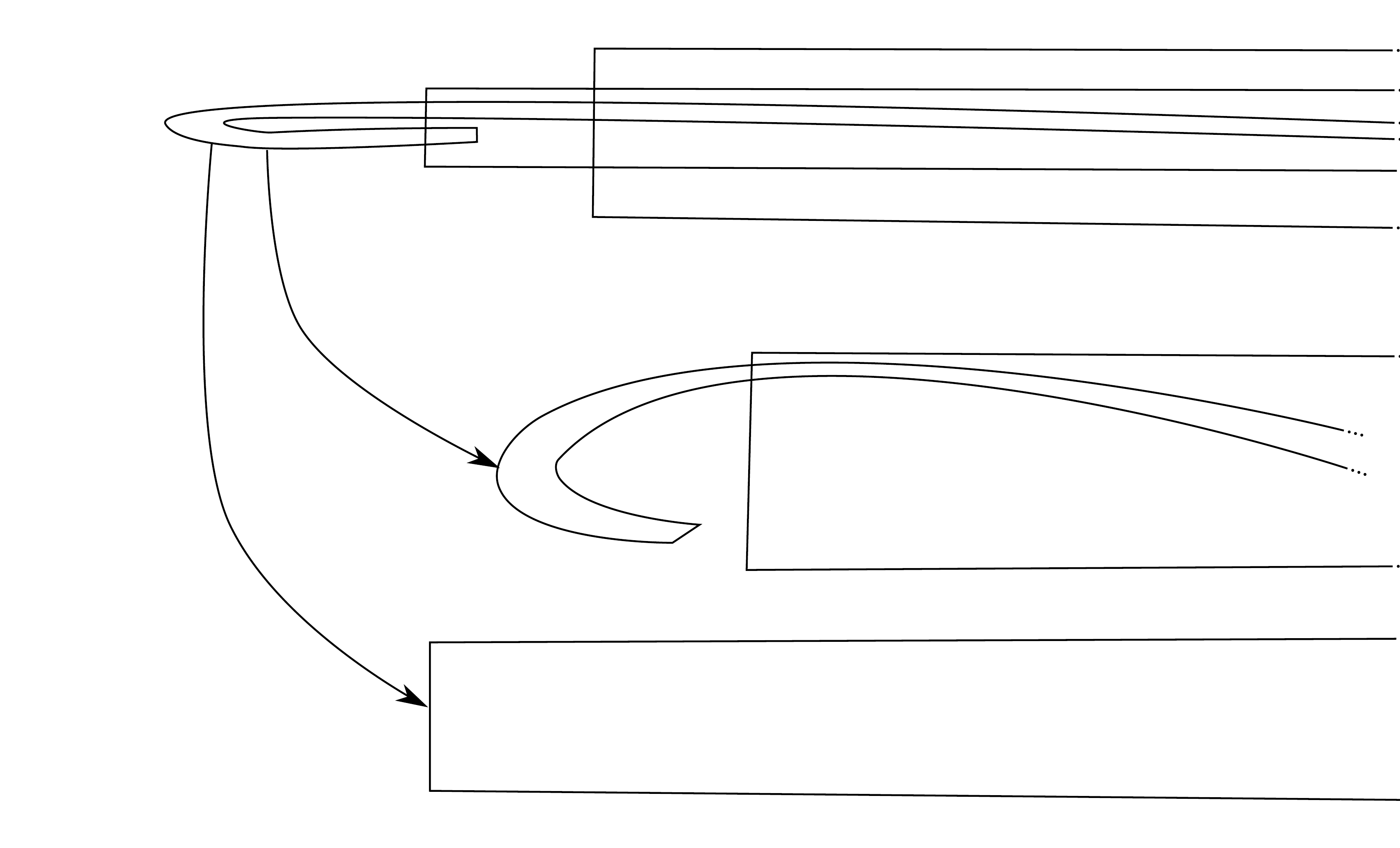}}%
    \put(-0.004,0.5){\color[rgb]{0,0,0}\makebox(0,0)[lb]{\small{$\tau_{n+1}(\s)$}}}%
    \put(0.3,0.45){\color[rgb]{0,0,0}\makebox(0,0)[lb]{\small{$\tau_n{\s}$}}}%
    \put(0.38,0.35){\color[rgb]{0,0,0}\makebox(0,0)[lb]{\small{$\tau_n(\sigma^m{\s})$}}}%
    \put(1.1,0.5){\color[rgb]{0,0,0}\makebox(0,0)[lb]{\small{$F_0$}}}%
    \put(1.1,0.25){\color[rgb]{0,0,0}\makebox(0,0)[lb]{\small{$F_{m}$}}}%
    \put(1.1,0.05){\color[rgb]{0,0,0}\makebox(0,0)[lb]{\small{$F_{m(n-1)}$}}}%
    \put(-0.002,0.22929324){\color[rgb]{0,0,0}\makebox(0,0)[lb]{\small{$f^{m(n-1)}$}}}%
    \put(0.23861859,0.36243614){\color[rgb]{0,0,0}\makebox(0,0)[lb]{\small{$f^m$}}}%
  \end{picture}%
\endgroup%

 %%%%%%%%%%%%%%%%%%%55
\end{center}
\caption{\small  
Let $\s=F_0F_1\ldots F_n\ldots\in\AA_B$ be an infinite address. The fundamental tail $\tau_{n+1}(\s)$ is contained in $\tau_{n}(\s)$ for points with large modulus. The fundamental tail $\tau_{n+1}(\s)$ is mapped to $\tau_{n}(\sigma^m\s)$ under $f^m$.
}
\label{fig:NestedTails}
\end{figure}
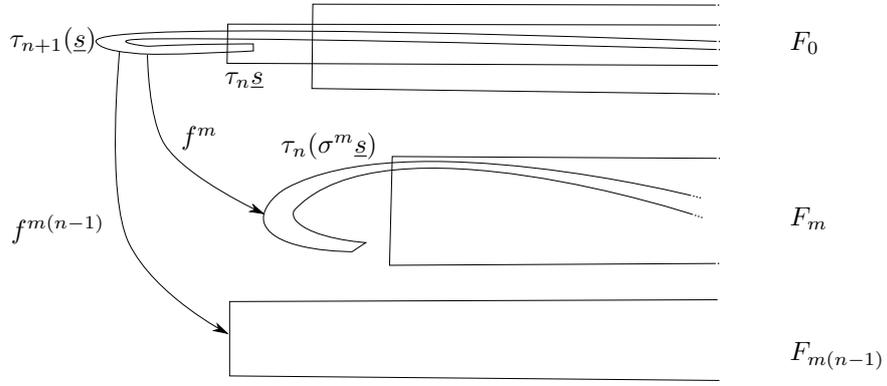

\begin{defn}[Pullback along an address]\label{defn:pullback along an address fp} Let $r>0$ and consider the fundamental tails  $\TT_1$ of level 1 for $r$. Let $\s=F_0F_1\ldots F_{\ell_n-1}\in\{\FF_B\}^{\ell_{n}}$ such that the tail $\tau_n(\s)$ exists for some $n$.   Let  $\zeta\in B_0$. 
When it exists,  we define $\zeta_n(\s)$ to be the unique point in $f^{-nm}(\zeta)\cap \tau_n(\s)$.  
\end{defn} 

Let  $\tau_n(\s)$ be a tail of level $n$. The map  $f^{m(n-1)}:\tau_n(\s)\ra\tau_1(\sigma^{m(n-1)}\s)$ is univalent, hence $f^{mn}$ is a univalent map from  $\tau_n(\s)$ to $\C\setminus (\ov{D_r}\cup\delta_r)$ (not necessarily surjective).  So, if $\zeta$ does have a preimage in $\tau_n(\s)$, such a preimage is unique.

\begin{lem}\label{lem:images of tails fp} A point $z\in\tau_n(\s)$ for some $n\geq1$ and some $\s\in\{\FF_B\}^{\ell_n}$ if and only if    $f^{jm}(z)\in \tau_{n-j}(\sigma^{mj}(\s))\subset B_0$ for all $j=1\ldots n-1$. 
\end{lem} 
 \begin{proof}
 This follows from the definition of fundamental tails.
 \end{proof}
 
We now show that up to taking $r$ large enough,   all possible fundamental tails of all levels exist unless $B=\bigcup_i B_i$ contains singular orbits which follow the itinerary of $z_0$ with respect to the partition into basic regions for $f^p$.

Recall that indices of the orbit of $z_0$ as well as indices of the basic regions containing them  are taken modulo $m$, and  that $f^j(z_0)\in B_j$ for all $j\geq0$.

Let $S_B(f)$ be the set of singular values which are contained in $B$, that is 
$$S_B(f)=S(f)\cap B.$$ For every  $s\in S_B(f)$ let $i(s)\in\{0,\ldots,m-1\}$ be such that $s\in B_{i(s)}$, and 
let 
$n(s)$ maximal be such that for all $0\leq j\leq n(s)$ we have that $f^j(s)\in B_{i(s)+j}$. Let
 \begin{equation}\label{eqtn:def of PB}\PP_B(f):=\ov{\bigcup_{s\in S_B(f)}\big(\bigcup_{ n\leq n(s)} f^n(s)\big)}.\end{equation}
Observe that $\PP_B(f)$ is smaller than $\PP(f)\cap B$, and that it  is forward invariant in the sense that
\begin{equation}\label{eqtn:forward invariance}
f(\PP_B(f)\cap B_i)\cap B_{i+1}\subset \PP_B(f).
\end{equation}
\begin{prop}[Existence of fundamental tails]\label{prop: eventually either univalent fp}
 Let $f\in\Brays$ such that periodic rays land. Let $\mathcal{X}=\{z_0,\ldots,z_{m-1}\}$ be a   repelling periodic orbit of  period $m$ and let $p$ be any multiple of $m$.  Suppose that $f(z_i)=z_{i+1} \mod m$. Let $\{B_i\}_{i=0\ldots m-1}$ be the basic regions for $f^p$ containing the elements of $\mathcal{X}$. Then  at least one of the following is true.
\begin{enumerate}
 \item[\rm (1)]\label{item:sv trapped fp} There exists a singular value $s$ for $f$ such that $s\in \bigcup_{i=0}^{m-1} B_i$,   say $s\in B_{i(s)}$, and such that for all $n\geq0$ we have that $f^{n}(s)\in B_{i(s)+n}$.  
\item[\rm (2)]\label{item:infinitely many sv trapped fp}   There are infinitely many singular values $s_j$ for $f$ in at least one of the basic regions $B_i$, say $  B_{0}$, and  a sequence $ n_j \ra  \infty$ as $j\ra\infty$ such that for all $n\leq n_j $ we have that  $f^{n}(s_j)\in B_{n}$.
\item[\rm (3)] The set $\PP(f)$ is bounded, and there exists $r>0$ such that all tails of all levels  are well defined for $z_0,p,r$. More precisely, this means the following. Let $\tau\in \TT_1(r)$, $n\geq0$, and $\tilde{\tau}$  a connected component of $f^{-nm}(\tau)$ for which $f^j(\tilde{\tau})\subset B_j$ for  $j\leq nm$. Then $f^{mn}:\tilde{\tau}\ra\tau$ is univalent.
\end{enumerate}
\end{prop}

\begin{proof}  
We first  claim that if neither case $\rm (1)$ nor case $\rm (2)$ occur, the set $\PP_B(f)$ is bounded. 
Indeed  we have that $$N:=\sup_{s\in S_B(f) } n(s)<\infty,$$ hence $$\PP_B(f)\subset \bigcup_{j\leq N}\ov{f^j(S_B(f))}$$ and each of the  sets $\ov{f^j(S_B(f))}$ is bounded  because $S_B(f)$ is bounded, and hence a  finite union thereof  is also bounded.

If $\PP_B(f)$ is bounded, let $r>0$ be  such that the tails of level $1$ for $D_r$ do not intersect $\PP_B(f)$. This can be done by taking $D_r \supset (D\cup \PP_B(f)\cup f(\PP_B(f))$ (notice that $f(\PP_B(f))$ is not contained in $B$). 

We claim that all tails of all levels are defined for such $r$.
Indeed, suppose that  this is not the case.  Then there exists a minimal $k>0$, a tail $\tau_{k}(\s)\in\TT_k(r)$, and a connected component $V\in f^{-m}(\tau_{k}(\s))\cap B_0$ such that $f^m : V\ra \tau_{k}(\s)$ is not univalent and such that $f^j(V)\in B_j$ (by definition of tails for $z_0$).   
 Since $\tau_{k}(\s)$ is simply connected, this occurs if and only if there exists  $j\leq m$ such that   $f^j(V)\subset B_j$   contains a singular value $s$. 

By definition of tails, there exists some $\tau\in \TT_1$ such that $f^{m(k-1)-j}(f^j(V))\subset\tau\subset B_0$, hence   $f^{m(k-1)-j}(s)\in\tau$. Since the orbit of $s$ follows the orbit of $f^j(V)$ (that is, $f^{\ell}(s)\in f^{\ell}(f^j(V))\subset B_{(\ell+j)}$) we have that  $f^{m(k-1)-j}(s)\in\PP_B$. 
This contradicts the  fact that by choice of $r$, $\TT_1\cap \PP_B(f)=\emptyset$.  
\end{proof}

Let us point out that $\PP(f)\cap B$ may well be unbounded even if $\PP_B$ is not. 
 
 Given Propostion~\ref{prop: eventually either univalent fp},   the strategy for proving the Main Theorem will be to show that in case (3), that is, in the absence of trapped singular values, every repelling periodic point is the landing point of  a periodic ray.

\section{Definition of landing and shrinking lemma} 
  
  In the current section we give an abstract definition of landing and we prove a lemma that will be used in Section~\ref{sect:main theorem and proof} (compare with the abstract characterization of landing in  \cite{BRG17}).

Let $z_0$ be a repelling periodic point of period $m$, $p$ be a multiple of $m$, $B$ and $B_i$ be basic regions for $f^p$ as in Section~\ref{sect:Fundamental tails for a repelling periodic orbit}. 
In this section we assume that we are in case (3) of Proposition~\ref{prop: eventually either univalent fp}, that is,  there exists $r>0$ such that all tails of all levels are well defined for $z_0, p, r$. In particular,  for any $\tau\in \TT_1$,  for every $n\geq0$, and for every connected component $\tilde{\tau}$ of $f^{-mn}(\tau)$ for which  $f^{j}(\tilde{\tau})\subset B_{j}$   for $j\leq n$, we have that   $f^{mn}:\tilde{\tau}\ra\tau$ is univalent.
 
 Let $\AA_B$ as in (\ref{eqtn:AB}) be the set of addresses for which tails  of all levels are well defined for $z_0,p,r$. 
 
The following lemma establishes one of the fundamental relations between rays and tails of the same address. 

\begin{lem}[Rays versus tails]\label{lem:rays and tails}  Let $\s=F_0 F_1 F_2\ldots\in\AA_B$. Then for every $z\in G_\s(t)$ there exists $n_z$ such that the arc connecting $z$ to infinity in $G_\s$  is fully contained in $\tau_n(\s)$ for all $n\geq n_z$.
\end{lem}

\begin{proof}  Let $D_r$ be the Euclidean disk of radius $r$ defined in the proof of Proposition~\ref{prop: eventually either univalent fp}, and consider the curve $\delta_r\subset\delta$ which starts from the last intersection of $\delta$ with $D_r$. Only for the proof of this lemma, let $\{F_i\}$ be the fundamental domains obtained by taking preimages of $\C\setminus (\ov{D_r}\cup\delta_r)$ and $\FF$ be the union  of all fundamental domains  for $f$ with respect to this choice of $D_r$.       By Lemma~\ref{lem:asymptotically contained}, the dynamic rays  $f^i(G_\s)$ are asymptotically contained in $F_i$ for all $i$.
 Let $G_\s(t):(0,\infty)\ra I(f)$ be a continuous parametrization of $G_\s$ such that $|G_\s(t)|\ra\infty$ as $t\ra\infty$.  Recall  that  points in $G_\s([T,\infty])$ escape uniformly to infinity for every $T>0$. So for each $z=G_\s(T)$ there exists $n_z$ such that
for all points     $G_\s(t)$ with $t>T$  we have 
\[f^{n}(G_\s(t))\in \FF \text{ for every  $n\geq n_z$}.\]
 (otherwise, $f^{n+1}(G_\s(t))$ would belong to the bounded set $D_r$, or to the curve $\delta_r$ which is mapped to $D_r$ at the next iterate,  contradicting the uniform escape to infinity).

 Hence we have that $f^n(G_\s(T,\infty))\subset F_n$ for $n\geq n_z$ and, by definition of tails, $G_\s(T,\infty)\subset \tau_{n}(\s) $ for all $n\geq n_z$.
\end{proof}
 
 The following Lemma is a Euclidean  version of  a classical Lemma which holds for  the spherical metric (see for example \cite{Lyu83}, Proposition 3. Similar lemmas have been used in   \cite{BL14},\cite{BRG17} and in many other papers). For this lemma we do not need the assumption that all  tails are well defined as long as we restrict to univalent preimages. 
    
    \begin{lem}[Shrinking Lemma]\label{lem:Shrinking Lemma Euclidean} Let $f$ be holomorphic.   
 Let $V'\subset \C$ be a simply connected  open set  intersecting the Julia set. Fix a compact set $K\subset\C$.  
 For each $n$ consider all connected components $V'_{n,\lambda}$ of $f^{-n}(V')$ which intersect $K$ and which are univalent preimages of $V'$ under $f^n$, where $\lambda$ indicates the chosen branch of $f^{-n}$.

Let $V\Subset V'$, and for each $n,\lambda$ let $V_{n,\lambda}=f^{-n}(V) \cap V'_{n,\lambda}$.   Then   for any $\epsilon >0$   there exists $N_\epsilon$ such that 

$$\diam_{\eucl} V_{n,\lambda}<\epsilon \text{\ \ \ \ for any  $n>N_\epsilon$ and for any $\lambda$ such that $V'_{n,\lambda}\cap K\neq\emptyset$}. $$
\end{lem}   
 The proof is the same as in \cite{Lyu83} for the spherical metric, and the statement for the Euclidean metric follows from the fact that we are only considering preimages intersecting a given compact set $K$.
 
\begin{lem}[Characterization of landing]\label{lem:characterization of landing} Let $f\in\Brays$ such that periodic rays land, and let $z_0, m, B_i, B,\AA_B$ as above. Let $G_\s\subset B_0$ be a periodic ray of period $mq$ with $\s\in\AA_B$ and $q\geq1$.    Let  $\zeta\in B_0 \setminus (\ov{D_r}\cup\delta_r)$ for which $\zeta_{nq}(\s)$ is well defined for all $n\in\N$ as in Definition~\ref{defn:pullback along an address fp}.  Then   $G_\s$ lands at $z_0\in\C$ if and only if $\zeta_{nq}(\s)\ra z_0$ as $n\ra\infty$.
\end{lem} 
 
Recall that a  dynamic ray  $G_\s$ \emph{lands} at a point $z\in\C$ if $G_\s(t)\ra z$ as $t\ra 0$.   
We observe the following. Let $z_0, m, r, \AA_B$ be as in the beginning of this section. Let $G_{\s}$ be a dynamic ray of period $mq$ for some $\s\in\AA_B$, $I$ be an arc in $G_\s$ connecting a point $z\in G_\s$ with its image $f^{mq}(z)$. For $n\in\N$ let $I_n$ be the unique preimage of $I$ under $f^{mqn}$ which is contained in $G_\s$. Then $G_{\s}$ lands at $z_0$ if and only if $\dist(I_n,z_0)\ra0$ as $n\ra\infty$. Indeed consider a sequence $G_\s(t_n)\ra z_0$. We have that $f^{k_n}(G_\s(t_n))\in I$ for some $k_n\ra\infty$ hence $G_\s(t_n)\in I_{k_n}$. 
  
Moreover if $I$   is chosen such that $f^j(I)$ is contained in $\FF$ for all $j$, then $I_n\subset \tau_{nq}(\s)$ for all $n\geq0$. The proof is the same as the proof of Lemma~\ref{lem:rays and tails}.
  
\begin{proof}[Proof of Lemma~\ref{lem:characterization of landing}] 
If $G_\s$ lands at $z_0$, by Lemma~\ref{lem:rays and tails} for any    point $\zeta=G_\s(t)$ with $t$ large enough  we have  that $\zeta_{nq}(\s)$ exists, and $\zeta_{nq}$ converges to $z_0$ by definition of landing.

To prove the other direction let  $\zeta\in B_0 \setminus (\ov{D_r}\cup\delta_r)$  such that $\zeta_{nq}(\s)$ is well defined (such a $\zeta$ exists  because $\s\in\AA_B$) and converges to $z_0\in\C$.  Let $\zeta'=G_\s(t)$ and let $I$ be the arc in $G_\s$ connecting $\zeta'$ to $f^{mq}(\zeta')$. Up to taking $t$ large enough we can assume that $f^j(I)\subset \FF$ for all $j$. 
 By Lemma~\ref{lem:rays and tails} we have that $I\subset\tau_{j}(\s)$ for all $j$ large enough.  By assumption, we also have that  $\zeta_{nq}(\s)\in\tau_{nq}(\s)$ for $n$ large enough. Hence $(I\cup \zeta_{nq})\subset\tau_{nq}(\s)$ for all $n$ large enough.     
For one such $n$ let   $V'\subset \tau_{nq}(\s)$ be a simply connected  set   containing both $\zeta_{nq}(\s)$ and $I$ and let $V\Subset V'$  containing both $\zeta_{nq}(\s)$ and $I$. Note that $V$ intersects the Julia set because dynamic rays are subsets of the Julia set. 

Let $K$ be a compact connected neighborhood of $z_0$.  For $j\in\N$ let $V'_{j}(\s) $ be the connected component of the  preimage of $V'$ under $f^{mqj}$ which is contained in $\tau_{q(j+n)}(\s)$; observe that the inverse branch $ \psi_j:V'\ra V'_{j}(\s)$ is univalent because $r$ was chosen so that all tails are well defined. Also, $V'_j$ intersects $K$ for $j$ large enough because $\zeta_{qn}(\s)\ra z_0$. 
By Lemma~\ref{lem:Shrinking Lemma Euclidean},  
\[\diam_{\eucl}V_{j}(\s)\ra0 \text{ as $j\ra\infty$}.\]
 Since we assumed that $f^j(I)\subset \FF$ for all $j$, $I_n\subset \tau_{nq}(\s)$, hence since $f^{nq}|_{\tau_{nq}(\s)}$ is a homeomorphism, $I_n\subset V_j\ra z_0$ and $G_{\s}$ lands at $z_0$.
\end{proof}

\section{Rationally invisible orbits and singular orbits}\label{sect:main theorem and proof}
  
The goal of this section is to prove that if tails are well defined for a given repelling periodic orbit with respect to a set of  basic regions containing it (case (c) in Proposition~\ref{prop: eventually either univalent fp}), then the orbit is not rationally invisible. 
The Main Theorem will be derived in Section~\ref{sect:FS}. 

In the following Theorem, as usual,   indices are taken modulo $m$. 
\begin{thm}[Main theorem for $f^p$]\label{thm:main theorem for fp}
Let $f\in\Brays$ such that periodic rays land and assume that there are no singular values escaping along  periodic rays. Let $\mathcal{X}=\{z_0,\ldots,z_{m-1}\}$ be a   repelling periodic orbit of  period $m$ and let $p$ be any multiple of $m$.  Suppose that $f(z_i)=z_{i+1}$. Let $\{B_i\}_{i=0\ldots m-1}$ be the basic regions for $f^p$ containing the elements of $\mathcal{X}$, and $B=\cup B_i$. Then  at least one of the following is true.    
\begin{enumerate}
\item[\rm (1)]\label{case: trapped sv final} There exists a singular value $s$ for $f$ such that $s\in \bigcup_{i=0}^{q-1} B_i$,   say $s\in B_{i(s)}$, and such that for all $n\geq0$ we have that $f^{n}(s)\in B_{n+i(s)}$.    
\item[\rm (2)]\label{case: infinitely many trapped sv final}  There are infinitely many singular values $s_j$ for $f$ in at least one of the basic regions $B_i$, say $s\in B_{i(s)}$, and  a sequence $ n_j \ra  \infty$ as $j\ra\infty$ such that for all $n\leq n_j $ we have that  $f^{n}(s)\in B_{n+i(s)}$. 
\item[\rm (3)]   Each point in $\XX$ is the landing point of at least one and at most finitely many periodic dynamic rays, all of which have the same period.   
\end{enumerate}
\end{thm}

By Lemma 8.2 in \cite{BRG17} (compare to \cite{Mi}, Lemma 18.12 for polynomials), if a repelling periodic point is the landing point of a periodic ray then it is the landing point of finitely many periodic rays, all of which have the same period. So it is enough to show that at least one point in $\XX$ is the landing point of at least one periodic dynamic ray. 
 This implies that  the same is true for all elements in $\XX$. Indeed, $f$ is a homeomorphism from a neighborhood of $z_i$ to a neighborhood of $z_{i+1}$,  so a dynamic ray $G_{\s}$ lands at $z_i$ if and only if  $f(G_\s)=G_{\sigma\s}$ lands at $z_{i+1}$. 

Let $z_0\in \XX$.   Recall the definition of fundamental tails for $z_0$ from Section~\ref{sect:Fundamental tails for a repelling periodic orbit}. 
By Proposition~\ref{prop: eventually either univalent fp}, if neither case (1) nor case (2) occur,      then there is $r$ such that all fundamental tails of all level are well defined. Our aim will be to show that under this assumption $z_0$ is the landing point of at least one periodic ray.

 Recall that $\ell_n=m(n-1)+1$ denotes the length of the address of a tail of level $n$.
  \begin{defn}[Fundamental pieces] Let $n\geq1$.   Let $\s$  be an infinite address or an address of length at least $\ell_{n+1}$ and assume that the tail $\tau_{n+1}(\s)$ is well defined for some $r>0$. Then we define the \emph{fundamental piece}  of level $n$ and address $\s$, which we denote by $P_n(\s)$, as 
\[P_n(\s):=\tau_{n+1}(\s)\setminus \tau_{n}(\s).\]
 \end{defn}
Fundamental pieces are not necessarily connected, nor exist for all levels and addresses. For example, if $\s$ is a disjoint-type address (i.e., contains only fundamental domains which do not intersect the disk $D$) then there are no fundamental pieces of address $\s$ for any level.

The idea of using fundamental pieces was originally suggested by L. Rempe-Gillen as a possible way to prove the main theorems in \cite{BRG17}.
   
Recall the definition of $\PP_B$ from Section~\ref{sect:Fundamental tails for a repelling periodic orbit}.  
   
   \begin{lem}[Properties of fundamental pieces]\label{lem:properties of fundamental pieces}
   Let $f\in\Brays$, such that periodic rays land. Let  $z_0,m, B, B_i$ as usual, and let $n\in\N$. 
 Let $\s$  be an infinite address or an address of length at least $\ell_n$ and assume that the fundamental tail $\tau_n(\s)$ is well defined.
Then  
\begin{equation}\label{eqtn:images of fundamental pieces}
 f^m(P_{j}(\s))= P_{j-1}(\sigma^m\s)\hspace{20 pt} \text{ for all } j\leq n-1  
\end{equation}  and 
 \begin{equation}\label{eqtn:strings of fundamental pieces}
 \tau_n(\s)\subset\tau_N(\s)\cup \bigcup_{j=N}^{n-1} P_j(\s)\hspace{20 pt}  \text{ for all } N\leq n-1.
\end{equation} 
\begin{equation}\label{eqtn:fundamental pieces mapping to disk}
f^{mn}(P_n(\s))=\tau_1(\sigma^{mn(\s)})\cap D \text{ for all $n\in\N$}.
\end{equation}
   \end{lem}
\begin{proof}  The first two properties follow from the definition of fundamental pieces and tails (recall that $f^m:\tau_j(\s)\ra \tau_{j-1}(\sigma^m \s)$ is a homeomorphism). We now prove (\ref{eqtn:fundamental pieces mapping to disk}). Let $\zeta\in P_n(\s)$. We have that   $f^{m(n-1)}(\zeta)\in \tau_2(\sigma^{m(n-1)}\s)\setminus \tau_1(\sigma^{m(n-1)}\s)$, and $\tau_1(\sigma^{m(n-1)}\s)$ is the preimage of $\C\setminus \ov{D_r}$. Hence $f^m(f^{m(n-1)}(\zeta))=f^{mn}(\zeta)\in D$. Since  $f^{m(n-1)}(\zeta)\in \tau_2(\sigma^{m(n-1)}\s)$,    $f^{mn}(\zeta)\in \tau_1(\sigma^{mn}\s)$ proving the claim. 
\end{proof}    
Recall that $S_B$ is the set of singular values which are contained in $B$. Recall also that  for $s\in S_B$ the integer  $i(s)\in\{0,\ldots,m-1\}$ is  such that $s\in B_{i(s)}$, and   $n(s)$ is maximal such that for all $j\leq n(s)$ we have that $f^j(s)\in B_{i(s)+j}$.  In other words the orbit of $s$ follows the orbit of $\XX$ for exactly $n(s)$ iterates with respect to the partition of the plane induced by the regions $B_i$.

\begin{lem}[Good neighborhoods of rays]\label{lem:good neighborhoods}
Let $f\in\Brays$ such that periodic rays land and assume that there are no singular values escaping along  periodic rays.   Let $\XX=\{z_0,\ldots,z_{m-1}\}$ be a   repelling periodic orbit   and let $p$ be any multiple of $m$.   Suppose that $f(z_i)=z_{i+1}$. Let $\{B_i\}_{i=0\ldots m-1}$ be the basic regions for $f^p$ containing the elements of $\mathcal{X}$, and let $B=\cup B_i$.

Suppose that cases (1) and (2) in Proposition~\ref{prop: eventually either univalent fp} do not hold. 

 Let $G$ be a ray in $\partial B_0$, which is hence fixed  under   $f^p$. 
Let $\{G_0=G, G_j=f^{j}(G)\}_{j=0,\ldots p-1}$ be the orbit of $G$ under $f$ (here, indices are taken modulo $p$), and let  $\psi_j: G_{j}\ra G_{j-1}$ be the  unique  the inverse branch of $f$  such that $\psi:=\psi_{0}\circ\ldots\circ\psi_{p-1}$ fixes $G$.
 Since $G_j$ are curves we can write them  as $G_j(t):\R_+\ra\C$, with $|G_j(t)|\ra\infty$ as $t\ra\infty$. Fix $\epsilon, T_j>0$.   
 
   Then there exist  neighborhoods $U_j$ of $G_j((\epsilon, T_j))$ which contain $G_j((0, T_j))$, which do not contain singular values for $f$,  and  such that:
   \begin{itemize}
   \item[a.] $\psi_j$ is defined and univalent on $U_j$; 
   \item[b.] $\psi_j(U_j)\subset U_{j-1}$;
   \item[c.] $\bigcup_j U_j\cap \bigcup_{s\in S_B}f^{n(s)+1}(s)=\bigcup_j U_j\cap (f(\PP_B)\setminus\PP_B)=\emptyset.$
\end{itemize}     
\end{lem} 

\begin{proof}
Notice first that we can take neighborhoods of $G_j((\epsilon, T_j))$ which do not intersect $S(f)$. If not there would be a sequence of singular values in $S(f)$ converging to a point $z\in  G_j([\epsilon, T_j])$, which would need to be a singular value because $S(f)$ is closed. This contradicts the  assumption that $G_j$ does not contain singular values.  This shows that we can take a neighborhood of $G_j([\epsilon, T_j])$ for every $\epsilon$. By taking the union over them, we obtain a neighborhood of $G_j((0, T_j])$.

Each $\psi_j$ is defined on compact subsets of $G_{j}$ containing the landing point, hence in particular, it is defined on $G([0, T_j])$.  
  Since $G_j$ contains no singular values for $f$ by assumption,  lands, and  $\psi_j(G_j)\subset G_{j-1}$,  for each $j$ there is a simply connected  neighborhood $U_j$ of  $G_j([\epsilon, T_j])$ which contains $G_j((0,T_j))$,  does not intersect the set $S(f)$ (recall that the latter is closed and that periodic rays contain no singular values), and such that $\psi_j$ is well defined on  $U_j$ and $\psi_j(U_j)\subset U_{j-1}$, with   $T_0> T$. 
  
  Let $N=\max_{s\in S_B}n(s)$. Since cases (1) and (2) in Proposition~\ref{prop: eventually either univalent fp} do not hold, $N<\infty$.
  
Since $G$ contains no points in singular orbits, $S_B$ does not intersect any preimage of $G$, and in particular it does not intersect the first $N$  preimages of $G$, that is the infinitely many rays preperiodic to $G$ of preperiod at most $N$.  Since $G$ lands and contains no points in singular orbits, they all land and only finitely many of them intersect any given compact set.  
    
 Let $D\Supset S(f)$ be a closed disk and let $K$ be the compact set given by the first $N$  preimages of rays in the boundary of $B$ intersected with $D$ (compare with the proof of Proposition~\ref{prop: eventually either univalent fp}).
 
Since $S_B$ does not intersect $K$ and both are compact sets, we can find a neighborhood $W$ of $K$ which does not intersect $S_B$, and restrict the sets $U_i$ such that their preimages up to level $N$ do not intersect $W$.
\end{proof} 
The condition that periodic rays do not contain points in singular orbits   can be relaxed by assuming that they do not contain forwards iterates of singular values in $S_B$ which moreover follow the correct itinerary between the basic regions in $B$.

\begin{lem}[Shrinking of fundamental pieces]\label{lem:shrinking of fundamental pieces}
Let $f\in\Brays$ such that periodic rays land and assume that there are no singular values escaping along  periodic rays. Let  $z_0,m, B, B_i$ as in Theorem~\ref{thm:main theorem for fp}. 
  Suppose that case $\mathrm{(1)}$ and $\mathrm{(2)}$ in Proposition~\ref{prop: eventually either univalent fp} do not hold. 
  
 Let $K$ be a compact set and consider all fundamental pieces $P_n(\s)$ for $n\in\N$. 
Then for any $\epsilon>0$ there exists $N_\epsilon$ such that 
\[\diam_{\eucl} P_n(\s)<\epsilon \text{ for all $n\geq N_\epsilon$ and all  $\s$ such that $P_n(\s)\cap K\neq\emptyset$.  }\]  
\end{lem} 
Compare with the proof of  Lemma 6.3 in \cite{BRG17}.

 \begin{proof} Since  case $\mathrm{(1)}$ and $\mathrm{(2)}$ in Proposition~\ref{prop: eventually either univalent fp} do not hold, $\PP_B$ is bounded, and so is its image $f(\PP_B)$ (which is no longer necessarily contained in $B$). Let    $D_r\Supset (\PP_B \cup \bigcup_{k\leq m} f^k(\PP_B))$  be a disk of radius $r$ centered at $0$. Consider the set of tails $\TT_1$ of level $1$ for $r$, $z_0$. Notice that $\PP_B$ is forward invariant in the sense of (\ref{eqtn:forward invariance}). It follows that $\tau_n(\s)\cap \PP_B=\emptyset$ for all $n\in\N$.
  
 For each of the finitely many $\tau\in\TT_1$ which intersect $D_r$ let $\gamma_\tau$ be a crosscut of $\tau$ such that $\tau\setminus\gamma_\tau$ is made of  two connected components, one of which is bounded and contains all of the connected components of $\tau\cap D_r$. Call this component $\eta_\tau$. 
Let 
$$\VV= {\bigcup_{\tau\in\TT_1}\eta_\tau}.$$

Since finitely many $\tau\in\TT_1$ intersect $D_r$ (see Lemma~\ref{lem:finitely many preimages intersect compact sets}) we have that $\VV$ has finitely many connected components.  
  Notice that if two adjacent tails $\tau,\tilde{\tau}$ both intersect $D_r$, their bounded components    $\eta_\tau, \eta_{\tilde{\tau}}$ belong to the same connected component of $\VV$. 
 
  We first  claim that   any $P_n(\s)$  is contained in a  connected components of $f_\lambda^{-n}(V_i)$ for some $i$ and some branch $\lambda$. 
   This is implied by showing that $f^{mn}(P_n(\s))\subset \eta_\tau$ for some $\tau\in\TT_1$.  
 Let   $\zeta\in P_n(\s)= \tau_{n+1}(\s)\setminus \tau_n(\s)$.  By (\ref{eqtn:fundamental pieces mapping to disk}) $f^{mn}(\zeta)\in\tau\cap D_r$ for some $\tau\in\TT_1$. The fact  that $\tau$ does not depend on the choice of $\zeta$ (we have to check this because $P_n(\s)$ is not necessarily a connected set) follows from the fact that for any $U$ connected component of the  preimage of $\eta_\tau$ under $f^{-mn}$ which is contained in a tail of level $n+1$ (which is the case for fundamental pieces) we have that $f^m:U\ra\eta_\tau$ is a homeomorphism.

So it is enough to show that, for $V$ which is  any of the finitely  many connected components of $\VV$,  the diameters of inverse images of $V$ tend to zero uniformly in the family of  inverse branches used to define fundamental pieces.  
Let $V$ be such a component. 
 To fix notation for the inverse branches let us denote by $\mathcal{L}$ the set of inverse branches  $\phi^n_\lambda$ of $f^{mn}$ which map a component $\eta_\tau\subset V$ inside another tail $\tau_{n-1}(\s)$, and which are a priori defined only on  $V$. 

We claim that   there is a simply connected neighborhood $V'$ of $V$  such that 
for any $\phi^n_\lambda\in\mathcal{L}$   we have that $\phi^n_\lambda$ can be extended (univalently) to $V'$. 

 The claim is obvious for all   $V\Subset B_0$, since by choice of $r$ we can find a simply connected neighborhood $V'$ which is contained in $B_0\setminus\PP_B$ and hence all inverse branches $\phi_\lambda^n$ used to define fundamental pieces can be extended.

So let us  consider $V$ such that $\partial V\cap \partial B_0\neq\emptyset$. 
 Let  $G$ on $\partial B_0$ which intersects $\partial V$, and $T$ such that $\partial V\cap G\Subset G(0,T)$. In the following we will assume for simplicity that $G$ is unique, but if not, there are finitely many rays and  the reasoning has to be done for each of them.  
   
Let $\{G_0=G, G_j=f^{j}(G)\}_{j=0,\ldots p}$ be the orbit of $G$ under $f$ and let  $\psi_j: U_{j}\ra U_{j-1}$  as in Lemma~\ref{lem:good neighborhoods}. 

 Let $V'\subset (U_0\cup B_0)$ be a   simply connected neighborhood of $V$ which does not intersect $\PP_B$, and let $\phi_\lambda^n\in\LL$. Choose $V'$ so that in addition $V'\cap B$ and $V'\cap (\C\setminus B)$ are simply connected.
 
The inverse branch  $\phi_\lambda^n$ decomposes (uniquely) as
$$
\phi_\lambda^n=h_{nm}\circ \ldots\circ h_1
$$
where  each $h_i$ for $i=1\ldots nm$ is a branch of $f^{-1}$ that we want to show to be defined on 
$h_{i-1}\circ \ldots\circ h_1(V')$. Notice that $\phi_\lambda^n$ extends to $V'\cap B$ because the latter does not intersect $\PP_B$.
 
 We claim that the inverse branch $h_1$ is well defined and univalent on all of $V'$. Let us denote by $X$ the connected component of $f^{-1}(V)$ which contains $h_1(V)$. Either $X\Subset B$  or $X\cap \partial B\neq \emptyset$. 
 
 If $X\Subset B$ then  the branch $h_1$ is well defined and univalent  because by Lemma~\ref{lem:good neighborhoods} the neighborhoods $U_j$ do not intersect $f(S_B)\subset f(\PP_B)$. Since they also do not intersect $f(\PP_B)$,  the set $X$ does not intersect $\PP_B$. Since the latter is forward invariant in the sense of (\ref{eqtn:forward invariance}), the branches $h_i$ are well defined also  for all $i=2\ldots mn$, proving the claim.
 
Let us consider the case   $X\cap \partial B\neq \emptyset$. Since $h_1(V)\subset B$, by  Lemma~\ref{lem:forward invariant boundary} we have that

\begin{align*}
 f(X\cap B) &= V'\cap B       \\
 f(X\cap (\C\setminus B)) &= V'\cap (\C\setminus B)      \\
 f(X\cap\partial B) &= V'\cap \partial B
\end{align*}
It follows that $X\cap\partial B$ is contained in $G_{p-1}$ which is the only preimage of $G_0$ in $\partial B$.
  
  By univalency of $f$ on $G_{p-1}$,  $h_1$ extends continuously to $G_0\cap V'$ and coincides with $\psi_0$ on this set. Since $h_1$ extends holomorphically to a neighborhood of $V\cap G_0$ (since $U_0$ contains no singular values by choice),  by the identity principle $h_1$ equals $\psi_0$  and hence $h_1$ extends as a univalent map on all of $V'$.
  
 By property b. in Lemma~\ref{lem:good neighborhoods}, $h_1(V'\setminus B)\subset U_{p-1}$. 
This last property allows us to repeat the reasoning for $h_2$ and show that it is defined on $h_1(V')$.  Proceeding by induction this gives the claim. 
  
 By Lemma~\ref{lem:Shrinking Lemma Euclidean}, $\diam_{\eucl}(\phi^n_\lambda(V))$ uniformly in $\lambda$ as $n\ra\infty$, provided $\phi^n_\lambda(V)\cap K\neq\emptyset$ for some compact set $K$. 
For all addresses  $\s$ such that $P_n(\s)\cap K\neq\emptyset$, the claim of the Lemma  follows because  $P_n(\s)\subset \phi_\lambda^n(V)$ for some $\lambda, n, i$.
\end{proof}

 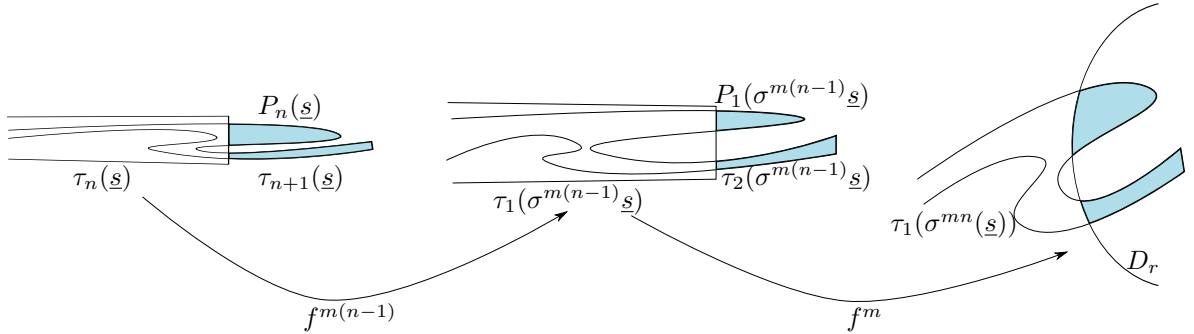
\begin{figure}[hbt!] 
\begin{center}
\def\svgwidth{\textwidth}
%%%%%
\begingroup%
  \makeatletter% 
    \setlength{\unitlength}{\svgwidth}%
  \makeatother%
   \begin{picture}(1,0.28077441)%
    \put(0,0){\includegraphics[width=\unitlength]{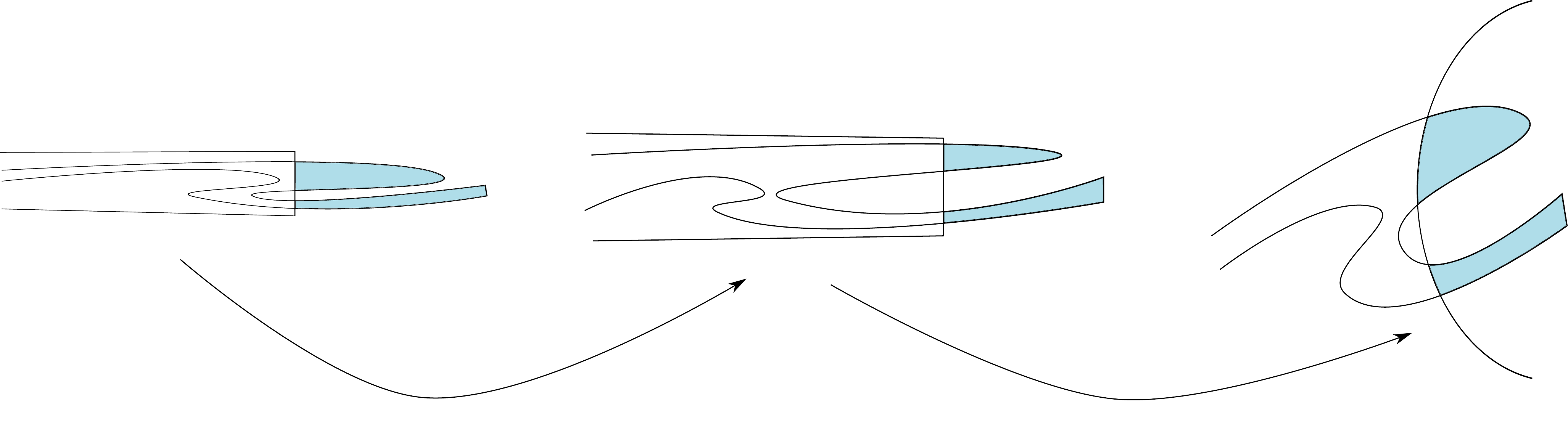}}%
    \put(0.25152843,0){\color[rgb]{0,0,0}\makebox(0,0)[lb]{\small{$f^{m(n-1)}$}}}%
    \put(0.71264726,0){\color[rgb]{0,0,0}\makebox(0,0)[lb]{\small{$f^m$}}}%
    \put(0.05590225,0.12){\color[rgb]{0,0,0}\makebox(0,0)[lb]{\small{$\tau_{n}(\s)$}}}%
    \put(0.21283314,0.12){\color[rgb]{0,0,0}\makebox(0,0)[lb]{\small{$\tau_{n+1}(\s)$}}}%
    \put(0.41275879,0.1){\color[rgb]{0,0,0}\makebox(0,0)[lb]{\small{$\tau_1(\sigma^{m(n-1)}\s)$}}}%
    \put(0.60623523,0.12){\color[rgb]{0,0,0}\makebox(0,0)[lb]{\small{$\tau_2(\sigma^{m(n-1)}\s)$}}}%
    \put(0.75,0.08){\color[rgb]{0,0,0}\makebox(0,0)[lb]{\small{$\tau_1(\sigma^{mn}(\s))$}}}%
    \put(0.95,0.05){\color[rgb]{0,0,0}\makebox(0,0)[lb]{\small{$D_r$}}}%
    \put(0.21283314,0.18){\color[rgb]{0,0,0}\makebox(0,0)[lb]{\small{$P_{n}(\s)$}}}%
    \put(0.6,0.19){\color[rgb]{0,0,0}\makebox(0,0)[lb]{\small{$P_1(\sigma^{m(n-1)}\s)$}}}%
  \end{picture}%
\endgroup%

 %%%%%%%%%%%%%%%%%%%55
\end{center}
\caption{\small Mapping properties of fundamental pieces (shown in blue).}
\label{fig:Shrinking Of Fundamental Pieces}
\end{figure}
The last result that we need in order to prove the Main Theorem is  Iversen's Theorem \cite{iversen}.  We state it as it is presented in \cite{BE08}. It is a consequence of Gross Star Theorem \cite{Gross}.    
 \begin{thm}[Iversen’s Theorem]Let $f$ be holomorphic. Let $\psi$ be a holomorphic branch of the inverse of $f$ which is  defined in a neighborhood of some point $z_0$ and let $\gamma:[0,1]\ra\C$  be a curve with
$\gamma(0)=z_0$. Then for every $\epsilon>0$  there exists a curve $\tilde\gamma:[0,1]\ra\C$ satisfying
$\gamma(0)=z_0$ and $|\gamma(t) - \tilde\gamma(t)| < \epsilon$ such that $\psi$ has an analytic continuation along $\tilde{\gamma}.$
 \end{thm}

 \subsection*{Proof of Theorem~\ref{thm:main theorem for fp}}

Let $z_0, z_i$, $m,p,B, B_i$ be as in the statement. Recall that indices of $z_i$ and $B_i$ are taken modulo $m$. 
  If $z_0$ is not an interior fixed point for $f^p$, there is nothing to show because it is the landing point of a periodic ray of period at most $p$. Otherwise, in view of Proposition~\ref{prop: eventually either univalent fp}  we need to show that, if there exists $r>0$ such that  for  any $n\geq 1$ all fundamental   tails for $z_0$ are well defined, then $z_0$ is the landing point of a periodic ray.  Without loss of generality up to taking a larger $r$ we can assume that $\bigcup_i z_i\subset D_r$ and that tails are also defined for a slightly smaller $r$. Let $\TT_n$ denote the collection of tails of level $n$ for these choices.
  
Let $\psi$ be the inverse branch of $f^{-m}$ fixing $z_0$. 
Let $U_0\subset B_0$ be a simply connected  neighborhood of $z_0$    such that  $\psi$ is well defined  on $U_0$,   $\psi (U_0)\Subset U_0$, and there exists  $\eta>1$ such that $|(f^m)'(z)|\geq \eta$ for all $z\in U_0$.  Note that $f^i(\psi(U_0))\subset B_i$ for $i\leq m$, and that more generally for $\ell\in\N$ we have $f^i(\psi^\ell(U_0))\subset B_i$ for $i\leq m\ell$.

Let
 $$U_n:=\psi^n (U_{0}), \ \ \epsilon=\dist(\partial U_0, \partial U_1).$$
  By choice of $U_0,$ $f^{mn}:U_n\ra U_0$ is a homeomorphism.

\begin{subclaim}$\TT_n \cap U_0\neq \emptyset$ for all   $n$ large enough.  
\end{subclaim}     
\begin{subproof} Let $F$ be a fundamental domain for $f$ which intersects $B_0$ and
choose   
 $\zeta\in F\cap B_0$ not an exceptional value. Then by Montel's theorem there exists $n$ large enough  so that $f^{nm}(U_0)\ni \zeta$.   Since $f$ is open, there exists $\epsilon'$ such that the Euclidean disk  $\D_{\epsilon'}(\zeta) \subset f^{nm}(U_0)\cap B_0\cap F$ and contains no exceptional values. 
 
 Let  $\gamma:[0,1]\ra B_0$  be a curve with
$\gamma(0)=z_0$, $\gamma(1)=\zeta$. Let $\psi^n$ be the inverse branch of $f^{nm}$ fixing $z_0$. By Iversen's theorem   there exists a curve $\tilde\gamma:[0,1]\ra B_0$ such that 
$\tilde{\gamma}(0)=z_0$,  $\tilde{\gamma}(1)\in \D_{\epsilon'}(\zeta)$ and  such that $\psi^n$ has an analytic continuation along $\tilde{\gamma}.$ Hence  $\tilde{\gamma}(1)\in f^{nm}(U_0)$.   
 
 Since  $\psi^n$ has an analytic continuation along $\tilde{\gamma}$ we can ensure that  the same is true for $\psi^j$ for $j\leq n$.  Since $\tilde{\gamma}\cap\partial B=\emptyset$ we have that   $\psi^j(\tilde{\gamma})\cap\partial B=\emptyset$ for $j\leq n$ (see Lemma~\ref{lem:forward invariant boundary}). Recall   that we have $f^i(\psi^\ell(U_0))\subset B_i$   for  all $\ell\in\N$ and all  $i\leq m$. Hence 
  for any $n\in\N$  we have that  $f^j\psi^n(\tilde{\gamma})\subset B_j$ for  $j\leq m$. This implies that the point   $w=\psi^n(\tilde{\gamma}(1))$  belongs to some tail  $\tau_{n+1}(\s)$ for some $\s$.  Since  $\tilde{\gamma}(1)\in f^{nm}(U_0)$, and $\psi^n$ is a homeomorphism, we also have that $w\in U_0$, proving the fact that   $\TT_n \cap U_0\neq \emptyset$ for   $n$ large enough.
       
      Observe that if $\TT_N \cap U_0\neq\emptyset$ for some $N$, then $\TT_n \cap U_0\neq\emptyset$ for all $n\geq N$, because the preimage  under $\psi$  of a point $\zeta$ which belongs to  a tail in $\TT_n $ intersecting $U_k$ is a point $\psi z$ in a tail in $\TT_{n+1} $ intersecting $U_{k+1}$ (see also Lemma \ref{lem:images of tails fp}).   
\end{subproof}

Let  $N$ be  such that for all $n\geq N$ we have that  $\TT_n \cap U_0\neq\emptyset $ and  that  $\diam P_n(\s)<\epsilon$ for all  $P_n(\s)$ intersecting $\ov{U_0}$.  Such $N$ exists by Lemma~\ref{lem:shrinking of fundamental pieces}. 
     
 For $n\geq N$ let $\SS_n$    be the set of finite addresses 
      of length   $\ell_n$   (see Definition~\ref{defn:addresses of tails})   for which the tail $\tau_{n}(\s)$ intersects $U_{n-N}$. Observe that  
     $\SS_n$ is       finite for every $n$ by Lemma~\ref{lem:finitely many preimages intersect compact sets}.

   Observe that $\dist(\partial U_n,z_0)\ra0$ as $n\ra\infty$ because in $U_0$ the map $\psi$ is contracting by a factor $\eta^{-1}<1$. So by Lemma~\ref{lem:characterization of landing} it is enough to find a periodic address $\s_*$ of
  period $mq$ %such that $\tau_n(\s_*)\subset B_0$ for all $n$
   and a point $\zeta\in U_0$ such that $\zeta_i(\s_*)$ is well defined for all $i$ and such that $\zeta_{qn-N}(\s_*)\in U_{qn-N}$ in order to ensure that $G_{\s_*}$ lands at $\lim_{n\ra\infty}\zeta_{qn-N}(\s_*)=z_0$.  
   
     We now claim that, for some $n_0$ large enough,  $\psi$ induces a well defined map $\Gamma$ from the finite set $\SS_{n_0}$ into itself, implying that  $\Gamma$ has a periodic point of some period $q$.  We do this in several steps.

      Let $n\geq N$,   $\s\in\SS_n$ and let $\tau_n(\s)$ be a tail in $\TT_n$ which intersects $U_{n-N}$. 
      For every point $\zeta\in \tau_n(\s)\cap U_{n-N}$ the point $\psi\zeta\in U_{n-N+1}$ is well defined and belongs to some tail of level $n+1$ and address $\tilde{\s}(\zeta)$ depending on $\zeta$. Indeed, $\tau_n(\s)\cap U_{n-N}$ may have several connected components, and it is not clear a priori that $\psi$ maps each of these connected components to the same tail of level $n+1$.  
      What is clear however is that for each such $\tilde{\s}=\tilde{\s}(\zeta)$ 
the tail $\tau_{n+1}(\tilde{\s})$ intersects $U_{n-N+1}$ and that $\sigma^m\tilde{\s}=\s$ (see Lemma~\ref{lem:images of tails fp}). 
 Recall that for an address $\tilde{\s}$ we denote by $\pi_n\tilde{\s}$ its first $\ell_n$ entries. 
       \begin{subclaim}\label{claim:main inductive step}
     $\pi_n\tilde{\s}$ belongs to $\SS_{n}$ regardless of the choice of $\zeta$. \end{subclaim}
  \begin{subproof} For an illustration of the proof of this claim see Figure~\ref{fig:main inductive step}.   Let $\psi\zeta\in U_{n-N+1}\cap \tau_{n+1}(\tilde{\s})$. Recall that  
$\tau_{n+1}(\tilde{\s})\subset \tau_{n}(\tilde{\s})\cup P_n(\tilde{\s})$ and that it is a connected set. If $\psi\zeta\in \tau_{n}(\tilde{\s}) $ it follows directly  that the first $\ell_n$ entries of  $\tilde{\s}$ are in $\SS_n$. 
      Otherwise $\psi\zeta\in P_n(\tilde{\s})\cap U_{n+1-N}$ hence 
       $f^{m(n-N)}\psi\zeta\in P_{N}(\sigma^{m(n-N)}\tilde{\s})\cap U_{1}$ (see  Lemma~\ref{lem:properties of fundamental pieces}).   By choice of $N$, $\diam P_{N}(\sigma^{m(n-N)}<\epsilon$, so  $P_{N}(\sigma^{m(n-N)}\tilde{\s})\Subset U_0$ and hence its $m(n-N)$-th pullback $ P_n(\tilde{\s})\Subset U_{n-N}$.  Since $\tau_{n+1}(\tilde{\s})$ is connected, $ \tau_{n}(\tilde{\s})$ intersects $U_{n-N}$ and hence the first $\ell_n$ entries of  $\tilde{\s}$ are in $\SS_n$.
      \end{subproof}

\begin{figure}[hbt!] 
\begin{center}
\def\svgwidth{0.8\textwidth}
 
\begingroup%
  \makeatletter% 
    \setlength{\unitlength}{\svgwidth}%
  \makeatother%
  \begin{picture}(1,0.46303444)%
    \put(0,0){\includegraphics[width=\unitlength]{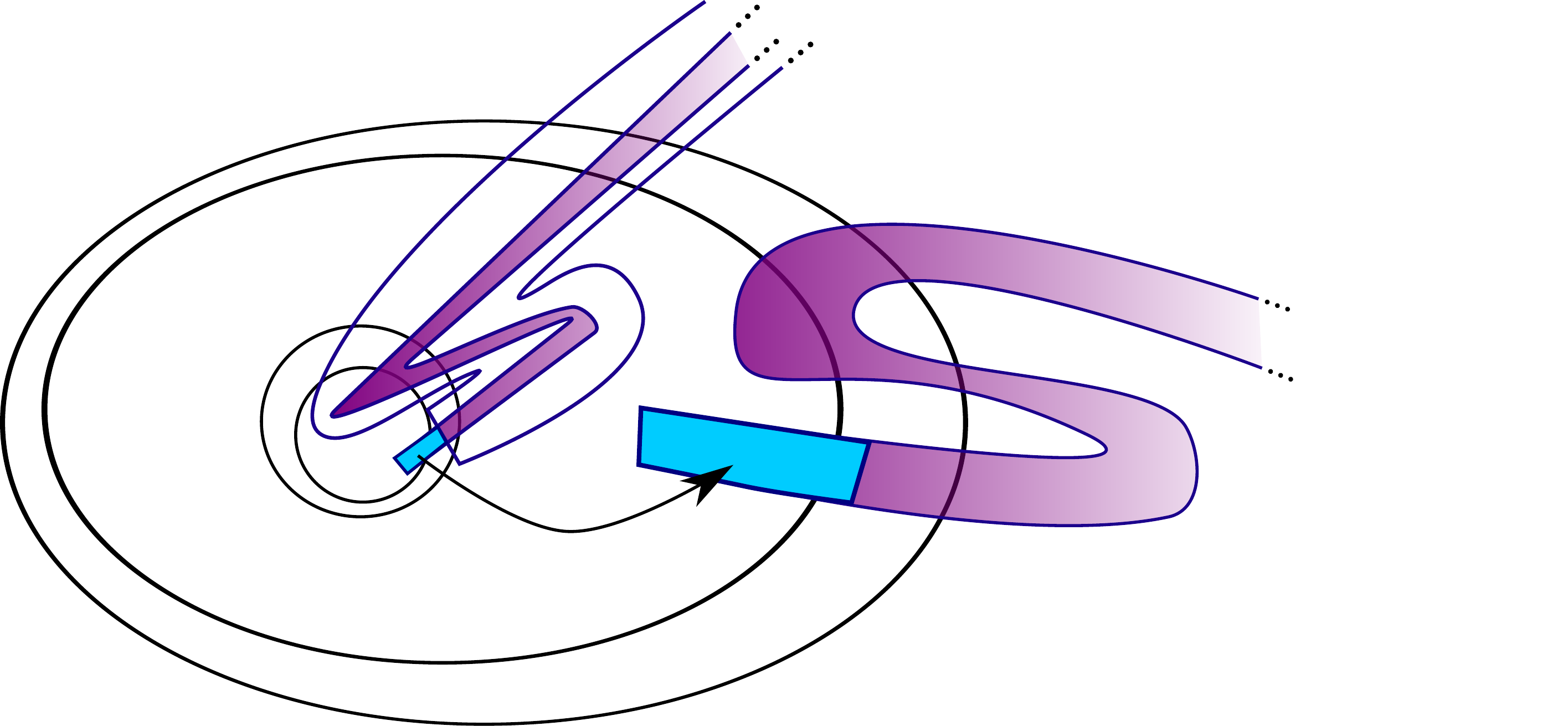}}%
    \put(0.51,0.1){\color[rgb]{0,0,0}\makebox(0,0)[lb]{\small{$U_1$}}}%
    \put(0.56,0.05){\color[rgb]{0,0,0}\makebox(0,0)[lb]{\small{$U_0$}}}%
    \put(0.07564212,0.1829143){\color[rgb]{0,0,0}\makebox(0,0)[lb]{\small{$U_{n-N}$}}}%
    \put(0.17274352,0.1){\color[rgb]{0,0,0}\makebox(0,0)[lb]{\small{$U_{N-n+1}$}}}%
    \put(0.75818295,0.09744429){\color[rgb]{0,0,0}\makebox(0,0)[lb]{\small{$\tau_{N+1}(\sigma^{m(n-N)})$}}}%
    \put(0.38678899,0.30030654){\color[rgb]{0,0,0}\makebox(0,0)[lb]{\small{$\tau_{n}(\s)$}}}%
    \put(0.51,0.4){\color[rgb]{0,0,0}\makebox(0,0)[lb]{\small{$\tau_{n+1}(\s)$}}}%
    \put(0.33841415,0.09){\color[rgb]{0,0,0}\makebox(0,0)[lb]{\small{$f^{mn}$}}}%
  \end{picture}%
\endgroup%
\end{center}
\caption{\small  }
\label{fig:main inductive step}
\end{figure}

 \begin{subclaim}\label{claim:eventual injectivity} There exists $n_0>0$ such that for $n\geq n_0$, the map $\sigma^m:\SS_{n+1}\ra\SS_{n}$ is injective, hence has a well defined inverse $\psi_n^*:\SS_{n}\ra\SS_{n+1}$ on its image $\sigma^m(\SS_{n+1})$.
 \end{subclaim} 
 
\begin{subproof}Observe that $z_0$ does not belong to any tail of level $N$, and hence  there   exists $M>0$ such that  $\TT_N\cap U_{M}=\emptyset$. Indeed, otherwise we would have that  $f^{mN}(z_0)=z_0\in f^{mN}(\TT_N)\subset \C\setminus D_r$, contradicting the choice of $r$. 
  The main point now  is to show that there exists $n_0$ such that if $n\geq n_0$ and  $\tilde{\s}\in\SS_{n+1}$ then $P_n(\tilde{\s})\subset U_1$.
Notice that a priori it is not even clear whether $P_n(\tilde{\s})\cap U_1\neq\emptyset$, because the intersection  $\tau_{n+1}(\tilde{\s})\cap U_{n+1-N}$ may be contained in $\tau_{n}(\tilde{\s})$. 
  Once this is proven,  let  $\s\in\SS_n$ and suppose by contradiction  that there exist $\tilde{\s}_1$, $\tilde{\s}_2$   such that $\psi(\tau_n(\s)\cap U_0)$ intersects both  $\tau_{n+1}(\tilde{\s}_1)$ and $\tau_{n+1}(\tilde{\s}_2)$. Then we would have that $P_{n}(\tilde{\s}_1)$ and $P_{n}(\tilde{\s}_2)$ are contained in $U_1$ and are mapped homeomorphically to $P_{n-1}(\s)$ by $f^m$ (recall Lemma~\ref{lem:properties of fundamental pieces}), contradicting the fact that $f^m: U_1\ra U_0$ is a homeomorphism and proving  Claim~\ref{claim:eventual injectivity}.  So we now proceed to prove that $P_n(\tilde{s})\subset U_1$ if $\tilde{s}\in\SS_{n+1}$ and $n\geq n_0$. 

Let $n\geq M+N$,  with $M$ as above  such that $\TT_N\cap U_M=\emptyset$, and consider $\s\in \SS_{n+1}$,  that is,   $\tau_{n+1}(\s)\cap U_{n+1-N}\neq\emptyset$. We claim that 
 there exists $\tilde{n}=\tilde{n}(\s)\in\{n-M,\ldots, n\}$  maximal such that $P_{\tilde{n}}(\s)\cap U_{n+1-N}\neq\emptyset$. This is to ensure that $\tau_{n+1}(\s)\cap U_{n+1-N}$ intersects a  fundamental piece of address $\s$ and of sufficiently high level, namely whose level tends to infinity as fast as $n$.
 
 Since $\TT_N\cap  U_{n+1-N}=\emptyset$ (because  $n\geq M+N$) and since $$\emptyset\neq\tau_{n+1}(\s)\cap U_{n+1-N}\subset(\tau_N(\s)\cup\bigcup_{j=N}^n P_j(\s))\cap U_{n+1-N}$$ (see   Lemma~\ref{lem:properties of fundamental pieces}), there is some $\tilde{n}\in\{N,\ldots, n\}$  maximal such that $P_{\tilde{n}}(\s)\cap U_{n+1-N}\neq\emptyset$.   
If $n=M+N$ this implies that $\tilde{n}\geq N= n-M$ as desired. Now let us proceed by induction on $n$. Suppose that for  all $\s\in\SS_n$ there exists $\tilde{n}=\tilde{n}(\s)\in\{n-1-M,\ldots, n-1\}$  maximal such that $P_{\tilde{n}}(\s)\cap U_{n-N}\neq\emptyset$ and  let us show that 
for  all $\s\in\SS_{n+1}$ there exists $\tilde{n}=\tilde{n}(\s)\in\{n-M-1,\ldots, n-1\}$  maximal such that $P_{\tilde{n}+1}(\s)\cap U_{n-N+1}\neq\emptyset$.
 Then this would imply the claim for $\tilde{n}+1$.  
If $\s\in\SS_{n+1}$ we have that $\sigma^m\s\in\SS_{n}$ and hence by the induction hypothesis we have that $P_{\tilde{n}}(\sigma\s)\cap U_{n-N}\neq\emptyset$ for some  $\tilde{n}=\tilde{n}(\sigma^m\s)\in\{n-1-M,\ldots, n-1\}$. Since $f^m:P_{\tilde{n}+1}(\s)\ra P_{\tilde{n}}(\sigma^m\s)$ is univalent we have that 
$P_{\tilde{n}+1}(\s)\cap U_{n+1-N}\neq\emptyset$ and that $\tilde{n}(\sigma\s)+1\geq n-N$ as required.
Now let $n_1$ such that $\diam P_n(\s)<\frac{\dist(\partial U_2,\partial U_1)}{M+1}$ for $n\geq n_1$ and $\s\in \SS_n$ (this is possible by Lemma~\ref{lem:shrinking of fundamental pieces}). 
 
Let $n_0=\max\{n_1+M,N+2\}.$ Let $n>n_0$ and let $\s\in\SS_{n+1}$. Then 
by the previous paragraph   there is $\tilde{n}\in \{n-M,\ldots, n\}$ such that $P_{\tilde{n}}(\s)\cap U_{n-N+1}\neq\emptyset$. By definition $P_{\tilde{n}}(\s)\subset \tau_{\tilde{n}+1}(\s)$, and $\tau_{n+1}(\s)\subset\tau_{\tilde{n}+1}(\s)\cup \bigcup_{\tilde{n}+1}^{n} P_j(\s)$. Since $\tau_{n+1}(\s)$ is connected and $ \bigcup_{\tilde{n}+1}^{n} P_j(\s)$ consists of at most $M$ pieces of diameter at most $\frac{\dist(\partial U_2,\partial U_1)}{M+1}$ and $\tau_{\tilde{n}+1}(\s)\cap U_2\neq\emptyset$ we deduce that $\tau_{{n}+1}(\s)\setminus \tau_{\tilde{n}+1}(\s)\subset U_1$.   
\end{subproof}

 For $n\geq n_0$  this induces a map $\psi_n^*:\SS_n\ra\SS_{n+1}$, where for $\s\in\SS_n$ we define $\psi_{n}^*(\s)$ as the unique element in $\SS_{n+1}$ such that $\sigma^m\psi_{n}^*(\s)=\s$. 
   Recall that $\FF_B$ is the collection of fundamental domains intersecting $B$ and observe that 
   $$\psi_{n}^*(\s)=: \alpha(\s^i)\s^i \text{\hspace{20pt} for some } \alpha(\s^i)\in \{\FF_B\}^{m},$$
    since we add $m$ symbols when we go backwards once.

 Let 
  $$
  \Gamma:=\pi_{n_0}\circ\psi_{n_0}^*:\SS_{n_0}\ra\SS_{n_0}.
  $$ 
Since $\SS_{n_0}$ has finitely many elements, there exists $q\in\N$ and $\s^0\in \SS_{n_0} $ such that $\Gamma^q(\s^0)=\s^0$.  
Let $\s^i=\Gamma^i(\s)$ (hence $\s^q=\s^0$).  By definition of $\Gamma$ and of $\alpha(\s^i)$ we have that  $\s^{i+1}=\Gamma(\s^i)=\pi_{n_0}(\alpha(\s^i)\s^i)$.
Let 
$$
\s^*:=\ov{\alpha(\s^q)\ldots\alpha(\s^1)}.
$$   Notice that $\s^*$ is a periodic address of period $qm$. %  and that  $\pi_{n_0}\s^*=\s^0$. Indeed $\s^0=\Gamma^q\s^0=\pi_{n_0} \psi_{n_0}^*(\s^{q-1})=\pi_{n_0}\alpha(\s^0)\s^{q-1}.$

\begin{subclaim}\label{claim:Nuria}
For any $n\geq n_0$ and for any $\s\in\SS_n$, 
$$\psi^*_{n}\s=\alpha(\pi_{n_0}\s)\s.$$
\end{subclaim}
\begin{subproof}
For $n=n_0$ this is true by definition, so suppose the claim is true for all $\s\in\SS_n$ and let us show the claim for all $\s\in\SS_{n+1}$. By definition of $\psi_{n+1}^*$, for $\s\in\SS_{n+1}$ we have $\psi_{n+1}^*\s=F_0\ldots F_{m-1}\s\in\SS_{n+2}$ for some $F_0,\ldots, F_{m-1}\in\FF_B$. 
By Claim~\ref{claim:main inductive step} we have that $\pi_{n+1}(F_0\ldots F_{m-1}\s)\in\SS_{n+1}$, hence $\pi_{n+1}(F_0\ldots F_{m-1}\s)=\psi_{n}^*\tilde{\s}$ for some   $\tilde{\s}\in\SS_{n}$. Hence %Since  $\pi_{n+1}(F_0\ldots F_{m-1}\s)=\psi_{n}^*\tilde{\s}$ 
we have that  $\tilde{\s}=\pi_n(\s)$. By the induction hypothesis,  $F_0\ldots F_{m-1}=\alpha(\tilde{\s})=\alpha(\pi_{n_0}\tilde{\s})=\alpha(\pi_{n_0}{\s})$.
 \end{subproof}
 
   Let $\zeta\in\tau_{n_0}(\s_*)\cap U_{n_0-N}$. By definition of $\psi_n^*$ and by Claim~\ref{claim:Nuria}  we have that $ \psi^{qn}(\zeta)\in\tau_{qn+n_0}(\tilde\s)$ where
 
  \[\tilde\s=\psi_{nq+n_0}^*\ldots\psi_{n_0}^*\s_*=\underbrace{\alpha(\s^q)\ldots\alpha(\s^1)\ldots \alpha(\s^q)\ldots\alpha(\s^1)}_{\text{repeated $n$ times}}\s_*=\s_*\] by Claim~\ref{claim:Nuria}, so that 
  $ \psi^{qn}(\zeta)\in\tau_{qn+n_0}( \s^*)$ and hence $ \psi^{qn}(\zeta)=\zeta_{qn}(\s^*)$. 
Then $G_{\s_*}$ lands at $z_0$ by Lemma~\ref{lem:characterization of landing}  because  $\zeta_{qn}(\s^*)=\psi^{qn}\zeta\in U_{qn-n_0}\ra z_0$.

We note the following Corollary of Theorem~\ref{thm:main theorem for fp}.
 
\begin{cor}\label{cor:fibers}Let $f\in\Brays$ such that periodic rays land and assume that there are no singular values escaping along  periodic rays.  Let $z_0$ be a rationally invisible repelling periodic point for $f$. Let $\{z_0,\ldots, z_{m-1}\}$ be the orbit of $z_0$ and let $X$ be the union of the  dynamical fibers of $z_0,\ldots, z_{m-1}$ as defined in \cite{RSbif}.  Then $X$ contains either a singular orbit, or infinitely many singular values whose orbits belong to the fiber for more and more iterations.
\end{cor} 
  
\section{Bound on the number of rationally invisible orbits and generalized Fatou-Shishikura inequality}\label{sect:FS}
This last section is devoted to the proof of the Main Theorem. 

Let us recall the Main Theorem from \cite{BF17}.  As usual indices are taken modulo $m$.

\begin{thm}[Singular orbits trapped in basic regions]\label{thm:Singular orbits trapped in basic regions}
 Let $f$ be an entire transcendental map in class $\Brays$ whose periodic  rays land.   
Let $\mathcal{X}$ be a cycle of Siegel disks, attracting basins, parabolic basins or Cremer points of period $m$ and let $p$ be any multiple of $m$.  Let $\{B_i\}_{i=0\ldots m-1}$ be the basic regions for $f^p$ containing the elements of $\mathcal{X}$. Then, up to relabeling the indices, at least one of the following is true.
\begin{enumerate}
\item[\rm (1)] There exists a singular value $s$ for $f$ such that $s\in \bigcup_{i=0}^{m-1} B_i$, say $s\in B_0$, and such that $f^{n}(s)\in B_n$ for all $n\in\N$. The orbit of $s$ accumulates either on the non-repelling cycle or on the boundary of the cycle of Siegel disks.

\item[\rm (2)]   There are infinitely many singular values $s_j$ for $f$ in at least one of the basic regions $B_i$, say $B_0$, and  a sequence $ n_j \underset{j\to\infty}{\longrightarrow}  \infty$ such that      $f^{n}(s_j)\in B_n$ for all $n\leq n_j $.  The orbits $\{f^{n}(s_j)\}_{j\in\N, n\leq n_j}$ accumulate either    on  the non-repelling interior cycle, or on the boundary of the associated Siegel disk. 
\end{enumerate}
The first case always occurs if $\mathcal{X}$ is attracting or parabolic or if $f$ has only  finitely many singular values. 

Moreover, in case (1), the orbit of $s$ does not accumulate on any other interior periodic cycle or  on any point on  the boundary of a Siegel disk $\Delta\notin \XX$ (provided the point is not on a periodic ray or a periodic point). 
\end{thm}

 We are now ready to prove the Main Theorem. %Theorem~\ref{thm:main thm intro}. 
 We remark that Theorems~\ref{thm:main theorem for fp} and Theorem~\ref{thm:Singular orbits trapped in basic regions} are stronger than the Main Theorem 
%  Theorem~\ref{thm:main thm intro}
   in that they do not require finiteness of the number of singular orbits which do not belong to attracting or parabolic cycles.
 
\begin{proof}[Proof of Main Theorem]%[Proof of Theorem~\ref{thm:main thm intro}]
Suppose by contradiction that there are at least $q+1$ cycles of Siegel disks, Cremer points, or rationally invisible  repelling periodic orbits (possibly infinitely many). Let $p$ be the product of their periods.
Each element in each of the $q+1$ cycles is fixed by $f^p$ hence belongs to a different basic region for $f^p$ by Theorem~\ref{thm:Separation Theorem Entire}. In particular,  there are $q+1$ disjoint collections of basic regions whose interior periodic orbit  is either a cycle of  Cremer point, a cycle of    centers of  Siegel disks, or a rationally invisible repelling periodic orbit. By Theorem~\ref{thm:main theorem for fp}, and since finitely many singular orbits which do not   belong to attracting or parabolic basins, we have that  each collection of basic regions which contains a rationally invisible repelling periodic orbit contains a singular orbit. By Theorem~\ref{thm:Singular orbits trapped in basic regions}, the same is true for  collections of basic regions which contain either Cremer points or centers of Siegel Disks. This gives $q+1$ singular orbits which do not belong to attracting or parabolic basins, a contradiction. 
\end{proof}

\bibliographystyle{amsalpha}
\bibliography{rationallyinvisible}
\end{document}